\documentclass[11pt,twoside, leqno]{article}

\usepackage{amssymb}
\usepackage{amsmath}
\usepackage{mathrsfs}
\usepackage{amsthm}
\usepackage{txfonts}
\usepackage{color}

\usepackage{indentfirst}

\allowdisplaybreaks

\def\ls{\lesssim}

\def\fz{\infty}

\renewcommand{\r}{\right}
\newcommand{\lf}{\left}

\def\ls{\lesssim}

\def\supp{{\mathop\mathrm{\,supp\,}}}

\def\rr{{\mathbb R}}

\def\rn{{{\rr}^n}}

\def\nn{{\mathbb N}}

\newcommand{\wz}{\widetilde}

\newcommand{\cm}{{\mathcal M}}

\newcommand{\cp}{{\mathcal P}}

\newcommand{\cs}{{\mathcal S}}

\def\az{\alpha}
\def\lz{\lambda}

\def\dz{\delta}

\def\bz{\beta}

\def\gz{{\gamma}}

\def\tz{\theta}

\def\wz{\widetilde}

\def\ls{\lesssim}

\def\boz{\Omega}

\def\uc{{\varepsilon}}

\def\esup{\mathop\mathrm{\,ess\,sup\,}}

\def\divz{{{\mathop\mathrm {div}}}}

\newtheorem{theorem}{Theorem}[section]
\newtheorem{lemma}[theorem]{Lemma}
\newtheorem{corollary}[theorem]{Corollary}

\theoremstyle{definition}
\newtheorem{remark}[theorem]{Remark}
\newtheorem{definition}[theorem]{Definition}
\def\supp{{\mathop\mathrm{\,supp\,}}}

\def\loc{{\mathop\mathrm{loc}}}

\numberwithin{equation}{section}

\begin{document}

\title{\Large\bf A Two-Weight Boundedness Criterion and Its Applications
\footnotetext{\hspace{-0.35cm} 2020 {\it Mathematics Subject
Classification}. {Primary 47A30; Secondary 42B20, 42B25, 42B37, 46E30, 42B35.}
\endgraf{\it Key words and phrases}. two-weight inequality, Lorentz space, Morrey space,
good-$\lambda$ inequality, Calder\'on--Zygmund operator,
Littlewood--Paley function, fractional integral, Riesz transform.
\endgraf This work is supported by the National Natural Science Foundation
of China (Grant Nos. 11871254, 12071431 and 11871100).}}
\author{Sibei Yang\,\footnote{Corresponding author,
E-mail: \texttt{yangsb@lzu.edu.cn}/{\color{red} April 26, 2021}/Final version.}\ \ and Zhenyu Yang}
\date{ }
\maketitle

\vspace{-0.8cm}

\begin{center}
\begin{minipage}{13cm}\small
{{\bf Abstract.} In this article, the authors establish a general (two-weight) boundedness criterion for
a pair of functions, $(F,f)$, on $\mathbb{R}^n$ in the scale of weighted Lebesgue spaces, weighted Lorentz spaces,
(Lorentz--)Morrey spaces, and variable Lebesgue spaces.
As applications, the authors give a unified approach to prove the (two-weight) boundedness of
Calder\'on--Zygmund operators, Littlewood--Paley $g$-functions, Lusin area
functions, Littlewood--Paley $g^\ast_\lambda$-functions, and fractional integral operators,
in the aforementioned function spaces. Moreover, via applying the above (two-weight) boundedness criterion,
the authors further obtain the (two-weight) boundedness of
Riesz transforms, Littlewood--Paley $g$-functions, and fractional integral operators associated with
second-order divergence elliptic operators with complex bounded measurable coefficients on $\mathbb{R}^n$
in the aforementioned function spaces.}
\end{minipage}
\end{center}

\vspace{0.2cm}

\section{Introduction\label{s1}}

The study of two-weight norm inequalities for some classical operators
in harmonic analysis is of interest not only for their own sake but also for
their many applications in partial differential equations (see, for instance,
\cite{cmp07,cp02,cp00,cp00b,m79,p95,p94,s88,sw92}).

Motivated by the work in \cite{a07,am07,sh07,sh05a}, in this article,
we establish a general two-weight boundedness criterion for a pair of functions,
$(F,f)$, on $\rn$ in the scale of both weighted Lebesgue spaces and weighted Lorentz spaces.
As applications, we further obtain a boundedness criterion for a pair of functions,
$(F,f)$, on $\rn$ in the scale of (Lorentz--)Morrey spaces and variable Lebesgue spaces.
Applying those (two-weight) boundedness criterion, we prove the (two-weight) boundedness of some classical
operators in harmonic analysis, including Calder\'on--Zygmund operators, Littlewood--Paley $g$-functions, Lusin area
functions, Littlewood--Paley $g^\ast_\lz$-functions, and fractional integral operators,
in the scale of weighted Lebesgue spaces, weighted Lorentz spaces, (Lorentz--)Morrey spaces,
and variable Lebesgue spaces. Although some of these boundedness are known, the proofs presented in this
article are quite different from those used in the existing literatures.
Moreover, using the general (two-weight) boundedness criterion given in this article,
we obtain the (two-weight) boundedness of Riesz transforms, Littlewood--Paley $g$-functions,
and fractional integral operators associated with second-order divergence elliptic operators with complex
bounded measurable coefficients on $\rn$ in the scale of weighted Lebesgue spaces,
weighted Lorentz spaces, (Lorentz--)Morrey spaces, and variable Lebesgue spaces.
It is worth pointing out that the boundedness of Riesz transforms, Littlewood--Paley $g$-functions,
and fractional integral operators associated with second-order divergence elliptic operators with complex
bounded measurable coefficients on $\rn$ in the scale of weighted Lebesgue spaces
obtained in this article extends the one-weight boundedness for the corresponding operators
established in \cite{am08,am06} to the two-weight case.

To describe some backgrounds and the main results of this article,
we first recall the definitions of weights, the Muckenhoupt weight class,
and the reverse H\"older class (see, for instance, \cite{cmp11,gr85,g14,St93}).

Let $E$ be a measurable subset in $\rn$ and $\omega$ be a weight function on $\rn$. Denote by $L^1_{\loc}(E)$ the
\emph{set of all locally integrable functions on $E$}. Furthermore, for any $f\in L^1_\loc(E)$,
we denote the integral $\int_E|f(x)|\omega(x)\,dx$ simply by
$\int_E|f|\omega\,dx$ and, when $|E|\in(0,\fz)$, we often use the following notation
$$\fint_E f\,dx:=\frac{1}{|E|}\int_Ef(x)\,dx.$$

\begin{definition}\label{d1.1}
\begin{itemize}
\item[\rm(i)] A measurable function $\omega$ on $\rn$ is called a \emph{weight} if $\omega$ is non-negative and locally integrable on $\rn$.
\item[\rm(ii)] Let $q\in[1,\fz)$. A weight $\omega$ on $\rn$ is said to belong to the \emph{Muckenhoupt weight class} $A_q(\rn)$, denoted by $\omega\in A_q(\rn)$, if, when $q\in(1,\fz)$,
\begin{equation*}
[\omega]_{A_q(\rn)}:=\sup_{B\subset\rn}\lf(\fint_B \omega\,dx\r)\lf(\fint_B
\omega^{-\frac{1}{q-1}}\,dx\r)^{q-1}<\fz
\end{equation*}
and, when $q=1$,
\begin{equation*}
[\omega]_{A_1(\rn)}:=\sup_{B\subset\rn}\lf(\fint_B \omega\,dx\r)
\lf\{\esup_{y\in B}[\omega(y)]^{-1}\r\}<\fz,
\end{equation*}
where the suprema are taken over all balls $B$ of $\rn$. Moreover, define  $$A_\fz(\rn):=\bigcup_{p\in[1,\fz)}A_p(\rn).$$

\item[\rm(iii)] Assume that $s\in(1,\fz]$. A weight $\omega$ on $\rn$
is said to belong to the \emph{reverse H\"older class} $RH_s(\rn)$,
denoted by $\omega\in RH_s(\rn)$, if, when $s\in(1,\fz)$,
\begin{align*}
[\omega]_{RH_s(\rn)}:=\sup_{B\subset\rn}\lf(\fint_B \omega^s\,dx\r)^{\frac1{s}}\lf(\fint_B\omega\,dx\r)^{-1}<\fz
\end{align*}
and, when $s=\fz$,
\begin{equation*}
[\omega]_{RH_\fz(\rn)}:=\sup_{B\subset\rn}\lf[\esup_{y\in
B}\omega(y)\r]\lf(\fint_B\omega\,dx\r)^{-1} <\fz,
\end{equation*}
where the suprema are taken over all balls $B$ of $\rn$.
\end{itemize}
\end{definition}

For the Muckenhoupt weight class and the reverse H\"older class, it is well known that
$$
A_\fz(\rn)=\bigcup_{p\in[1,\fz)}A_p(\rn)=\bigcup_{s\in(1,\fz]}RH_s(\rn)
$$
(see, for instance, \cite[Chapter 7]{g14}).

\begin{definition}\label{d1.2}
Let $\gz\in[0,1)$. Then the \emph{fractional Hardy--Littlewood maximal operator}
$\cm_\gz$ on $\rn$ is defined by setting, for any $f\in L^1_{\loc}(\rn)$ and $x\in\rn$,
$$\cm_\gz(f)(x):=\sup_{B\ni x}\lf(|B|^{\gz}\fint_{B}|f|\,dy\r),$$
where the supremum is taken over all balls $B$ of $\rn$ containing $x$.

In particular, when $\gz=0$, the fractional Hardy--Littlewood maximal operator
$\cm_\gz$ is just the well-known \emph{Hardy--Littlewood maximal operator}; in this case,
we denote $\cm_\gz$ simply by $\cm$.
\end{definition}

Assume that $p\in(0,\fz)$ and $\omega$ is a weight on $\rn$.
Recall that the \emph{weighted Lebesgue space} $L^p_\omega(\rn)$ is defined by setting
\begin{align*}
L^p_\omega(\rn):=\lf\{f\ \text{is measurable on}\ \rn: \
\|f\|_{L^p_\omega(\rn)}:=\lf[\int_{\rn}
|f|^p\omega\,dx\r]^{\frac1p}<\fz\r\}.
\end{align*}
In particular, when $\omega\equiv1$, the weighted space $L^p_\omega(\rn)$
is denoted simply by $L^p(\rn)$, which is just the classical Lebesgue space.

In what follows, for any given $q\in[1,\fz]$, we denote by $q'$
its \emph{conjugate exponent}, namely, $1/q+1/q'= 1$. Let $\omega$ be a weight on $\rn$
and $E$ a measurable subset of $\rn$.
Denote the integral $\int_E\omega(x)\,dx$ simply by $\omega(E)$.
%Furthermore, for any $f\in L^1_\loc(E)$, we denote the integral $\int_E|f(x)|\omega(x)\,dx$ by
%$\int_E|f|\omega\,dx$ and, when $|E|\in(0,\fz)$, we use the notation
%$$\fint_E fdx:=\frac{1}{|E|}\int_Ef(x)dx.$$

Muckenhoupt \cite{m72} proved that, for any given $p\in(1,\fz)$,
the Hardy--Littlewood maximal operator $\cm$ is bounded on the weighted Lebesgue space
$L^p_\omega(\rn)$ if and only if $\omega\in A_p(\rn)$. However, in the two-weight case, for any
given $p\in(1,\fz)$, the analogous two-weight $A_p(\rn)$ condition,
\begin{equation}\label{1.1}
\lf[\omega,v^{1-p'}\r]_{A_p(\rn)}:=\sup_{B\subset\rn}\lf(\fint_B \omega\,dx\r)
\lf(\fint_B v^{1-p'}\,dx\r)^{p-1}<\fz
\end{equation}
with the supremum taken over all balls $B$ of $\rn$,
is necessary but not sufficient for $\cm$ to be bounded from $L^p_v(\rn)$ to $L^p_\omega(\rn)$
(see, for instance, \cite{gr85,m72}). Meanwhile, Muckenhoupt \cite{m72} showed that, for any given $p\in(1,\fz)$,
the weak type inequality
$$\sup_{\lz\in(0,\fz)}\lf\{\lz\lf[\int_{\{x\in\rn:\ \cm(f)(x)>\lz\}}\omega\,dx\r]^{\frac{1}{p}}\r\}
 \le C\|f\|_{L^p_v(\rn)}
$$
with $C$ being a positive constant independent of $f$, holds true for any $f\in L^p_v(\rn)$
if and only if $[\omega,v^{1-p'}]_{A_p(\rn)}<\fz$.
Moreover, Sawyer \cite{s82} proved that, for any given $p\in(1,\fz)$,
$\cm$ is bounded from $L^p_v(\rn)$ to $L^p_\omega(\rn)$ if and only if $[\omega,v^{1-p'}]_{S_p(\rn)}<\fz$,
where
$$\lf[\omega,v^{1-p'}\r]_{S_p(\rn)}:=\sup_{B\subset\rn}\lf[\frac{1}{v^{1-p'}(B)}\int_B
\lf[\cm\lf(v^{1-p'}\mathbf{1}_B\r)\r]^p\omega\,dx\r]^{\frac1p}
$$
with the supremum taken over all balls $B$ of $\rn$. Here and thereafter, $\mathbf{1}_B$
denotes the \emph{characteristic function} of $B$. It is known that, if $[\omega,v^{1-p'}]_{S_p(\rn)}<\fz$, then
$[\omega,v^{1-p'}]_{A_p(\rn)}\le[\omega,v^{1-p'}]_{S_p(\rn)}^p<\fz$
(see, for instance, \cite{s82}). Conversely, if $[\omega,v^{1-p'}]_{A_p(\rn)}<\fz$ and $v$ satisfies
$v^{1-p'}\in A_\fz(\rn)$, then $[\omega,v^{1-p'}]_{S_p(\rn)}<\fz$ (see, for instance, \cite[Corollary 1.3]{p95} and \cite[Corollary 1.4]{pr15}).
A defect of the condition $[\omega,v^{1-p'}]_{S_p(\rn)}<\fz$ is that it involves the maximal operator $\cm$ itself
which hence motivates a search for sufficient conditions that are simpler and similar to
the two-weight $A_p(\rn)$ condition \eqref{1.1}. For this question, Neugebauer \cite{n83}
showed that, for any given $p\in(1,\fz)$, if the weights $\omega$ and $v$ satisfy $[\omega^r,v^{(1-p')r}]_{A_p(\rn)}<\fz$
for some $r\in(1,\fz)$ (which is called a \emph{``power bump" condition}),
then $\cm$ is bounded from $L^p_v(\rn)$ to $L^p_\omega(\rn)$. The result of Neugebauer \cite{n83}
was improved by P\'erez \cite{p95} via proving that the ``power bump" condition can be replaced
by a weak ``Orlicz bump" condition. To describe the result of P\'erez \cite{p95}, we first
recall some necessary definitions and notation.

\begin{definition}\label{d1.3}
\begin{itemize}
\item[\rm(i)] A function $\Phi:\ [0,\fz)\to[0,\fz)$ is called a \emph{Young function} if $\Phi$ is continuous, convex,
strictly increasing, $\Phi(0)=0$, and $\frac{\Phi(t)}{t}\to\fz$ as $t\to\fz$. Moreover, it is said that a Young function
$\Phi$ is \emph{doubling} if there exists a positive constant $C$ such that, for any $t\in[0,\fz)$, $\Phi(2t)\le C\Phi(t)$.

\item[\rm(ii)] Let $\Phi$ be a Young function and $B$ a ball in $\rn$. For any $f\in L^1_{\loc}(\rn)$,
the \emph{normalized Luxembourg norm} $\|f\|_{\Phi,\,B}$ of $f$ on $B$ is defined by setting
$$\|f\|_{\Phi,\,B}:=\inf\lf\{\lz\in(0,\fz):\ \fint_B\Phi\lf(\frac{|f(x)|}{\lz}\r)\,dx\le1\right\}.
$$
\end{itemize}
\end{definition}

\begin{remark}\label{r1.1}
Let $p\in(1,\fz)$, $B$ be a ball in $\rn$, and $\Phi(t):=t^p$ for any $t\in[0,\fz)$. Then
$\Phi$ is a Young function and, for any $f\in L^1_\loc(\rn)$,
$$\|f\|_{\Phi,\,B}=\lf(\fint_B|f|^p\,dx\r)^{\frac1p}=:\|f\|_{p,\,B}.
$$
\end{remark}

P\'erez \cite{p95} showed that, for any given $p\in(1,\fz)$, if $\Phi$ is a doubling Young function satisfying
\begin{equation}\label{1.2}
\int_c^\fz\lf[\frac{t^{p'}}{\Phi(t)}\r]^{p-1}\frac{dt}{t}<\fz
\end{equation}
for some constant $c\in(0,\fz)$, and the weights $\omega$ and $v$ satisfy
\begin{equation}\label{1.3}
\sup_{B\subset\rn}\lf(\fint_B \omega\,dx\r)\lf\|v^{-\frac1p}\r\|_{\Phi,\,B}^p<\fz
\end{equation}
with the supremum taken over all balls $B$ of $\rn$,
then $\cm$ is bounded from $L^p_v(\rn)$ to $L^p_\omega(\rn)$. Meanwhile, observe
that the condition \eqref{1.3} is sharp in the sense that, if \eqref{1.3} implies the
boundedness of $\cm$ from $L^p_v(\rn)$ to $L^p_\omega(\rn)$, then \eqref{1.2} holds true
(see, for instance, \cite{cmp11,cp02,p95}).
A typical example of $\Phi$ satisfying \eqref{1.2} is that $\Phi(t):=t^{s_0p'}$, with some $s_0\in(1,\fz)$,
for any $t\in[0,\fz)$. Two important examples of $\Phi$ satisfying \eqref{1.2} are as follows:
for any given $\dz\in(0,\fz)$ and for any $t\in[0,\fz)$,
$$\Phi(t):=t^{p'}[\log(e+t)]^{p'-1+\dz}
$$
and
$$\Phi(t):=t^{p'}[\log(e+t)]^{p'-1}\lf[\log\log\lf(e^e+t\r)\r]^{p'-1+\dz}$$
(see, for instance, \cite{cmp11,cmp07,crv14} for more details).
Moreover, it is worth pointing out that the two-weight boundedness of
some classical operators in harmonic analysis, including fractional maximal operators,
fractional integral operators, Calder\'on--Zygmund operators, and Littlewood--Paley square functions,
has been extensively studied in the last 30 years (see, for instance, \cite{co21,cmp11,cmp07,crv14,cp00,ll14,mw76,r98b,p94,sw92}
and the references therein).

Furthermore, to study the weighted norm inequalities of non-standard Calder\'on--Zygmund
operators in the scale of weighted Lebesgue spaces and the regularity of the solution
of second-order elliptic boundary value problems in non-smooth domains, Auscher \cite{a07},
Auscher and Martell \cite{am07}, and Shen \cite{sh05a} (see also \cite{sh20,sh18,sh07}) developed a (weighted) $L^p$-boundedness
criterion, which is inspired by the work of Caffarelli and Peral \cite{cp98}.

Motivated by the (weighted) $L^p$-boundedness criterion established in \cite{a07,am07,sh07,sh05a},
we obtain a general (two-weight) boundedness criterion for a pair of functions,
$(F,f)$, on $\rn$ in the scale of weighted Lebesgue spaces, weighted Lorentz spaces,
(Lorentz--)Morrey spaces, and variable Lebesgue spaces.

Now, we state the main results of this article as follows.

\begin{theorem}\label{t1.1}
Let $\gz\in[0,1)$, $p_1,\,p_2,\,p_3\in(0,\fz]$ satisfy $p_3>\max\{p_1,\,p_2\}$, and $F,\,f\in L^{1}_{\loc}(\rn)$.
Assume that, for any ball $B$ of $\rn$, there exist two measurable functions $F_B$ and $R_B$ on $B$
such that $|F|\le|F_B|+|R_B|$ on $B$, and
\begin{align}\label{1.4}
\lf(\fint_{B}|R_B|^{p_3}\,dx\r)^{\frac1{p_3}}
\le C_1\lf\{\lf[\cm(|F|^{p_1})(x_1)\r]^{\frac{1}{p_1}}+\lf[\cm_\gz(|f|^{p_2})(x_2)\r]^{\frac{1}{p_2}}\r\}
\end{align}
(with the usual modification made when $p_3=\fz$)
and
\begin{align}\label{1.5}
\lf(\fint_{B}|F_B|^{p_1}\,dx\r)^{\frac1{p_1}}
\le \epsilon{\lf[\cm(|F|^{p_1})(x_1)\r]}^{\frac{1}{p_1}}+C_2\lf[\cm_\gz(|f|^{p_2})(x_2)\r]^{\frac{1}{p_2}}
\end{align}
for any $x_1,\,x_2\in B$,
where $C_1$, $C_2$, and $\epsilon$ are positive constants independent of $F,\,f,\,R_B,\,F_B,$ and $B$.
Assume further that $\omega\in RH_s(\rn)$ with some $s\in(1,\fz]$, and $a\in(1,\frac{p_3}{p_1})$.
Then there exists a positive constant $\bz_0\in[1,\fz)$,
depending only on $C_1,\,C_2,\,n,\,p_1,\,p_2,\,p_3$, $a$, and
$[\omega]_{RH_s(\rn)}$, such that, for any given $\bz\in[\bz_0,\fz)$,
there exist an $\uc_0\in(0,\fz)$ and a $\kappa_0\in(0,1)$, depending only on $C_1,\,C_2,\,n,\,p_1,\,p_2,\,p_3$, $a$,
$[\omega]_{RH_s(\rn)}$, and $\bz$, such that, if $\uc\in[0,\uc_0)$ and $\kappa\in(0,\kappa_0)$, then, for any $\lz\in(0,\fz)$,
\begin{align}\label{1.6}
\omega(E(\bz \lz))\le\bz^{-\frac{(s-1)a}{s}}\omega(E(\lz))+\omega\lf(\lf\{x\in\rn:\ \cm_\gz(|f|^{p_2})(x)>
(\kappa\lz)^{\frac{p_2}{p_1}}\r\}\r),
\end{align}
where, for any given $\lz\in(0,\fz)$,
$$E(\lz):=\lf\{x\in\rn:\ \cm(|F|^{p_1})(x)>\lz\r\}.
$$
\end{theorem}

Theorem \ref{t1.1} gives a weighted good-$\lz$ inequality for a pair of functions, $(F,f)$,
on $\rn$ satisfying the assumptions \eqref{1.4} and \eqref{1.5}. We prove Theorem \ref{t1.1} via borrowing some
ideas from the proofs of \cite[Proposition 1.5]{a07}
(see also \cite[Theorem 3.1]{am07}) and \cite[Theorem 3.2]{sh07} (see also \cite[Theorem 4.2.3]{sh18}).
More precisely, we show Theorem \ref{t1.1} via using the Whitney covering lemma
(see, for instance, \cite[p.\,15, Lemma 2]{St93} or Lemma \ref{l3.2} below), the (weak) boundedness
of the Hardy--Littlewood maximal operator $\cm$ on Lebesgue spaces, and some properties of
the reverse H\"older class.

To describe a general (two-weight) boundedness criterion for a pair of functions, $(F,f)$,
on $\rn$ considered in this article, we recall the definitions of weighted Lorentz spaces,
(Lorentz--)Morrey spaces, and variable Lebesgue spaces as follows.

\begin{definition}\label{d1.4}
Assume that $q\in(0,\fz)$, $t\in(0,\fz]$, and $\omega$ is a weight on $\rn$.
The \emph{weighted Lorentz space} $L^{q,\,t}_\omega(\rn)$ is defined by setting
$$L^{q,\,t}_\omega(\rn):=\lf\{f\ \text{is measurable on}\ \rn:\
\|f\|_{L^{q,\,t}_\omega(\rn)}<\fz\r\},
$$
where, when $t\in(0,\fz)$,
$$\|f\|_{L^{q,\,t}_\omega(\rn)}:=\lf\{q\int_0^\fz
\lf[\lz^q\omega\lf(\lf\{x\in\rn:\ |f(x)|>\lz\r\}\r)\r]^{\frac tq}\frac{d\lz}{\lz}\r\}^{\frac1t},
$$
and
$$\|f\|_{L^{q,\,\fz}_\omega(\rn)}:=\sup_{\lz\in(0,\fz)}\lf\{\lz\lf[\omega\lf(\{x\in\rn:\
|f(x)|>\lz\}\r)\r]^{\frac1q}\r\}.
$$
\end{definition}

It is easy to see that, when $q=t$, $L^{q,\,t}_\omega(\rn)=L^q_\omega(\rn)$.
If $\omega\equiv1$, then we write $L^{q,\,t}_\omega(\rn)=L^{q,\,t}(\rn)$.

\begin{remark}\label{r1.2}
Let $q\in(0,\fz)$, $t\in(0,\fz]$, $s\in(0,\fz)$, and $\omega$ be a weight on $\rn$.
Then, for any $f\in L^1_\loc(\rn)$, if $|f|^s\in L^{q,\,t}_\omega(\rn)$,
then $\||f|^s\|_{L^{q,\,t}_\omega(\rn)}=\|f\|^s_{L^{qs,\,ts}_\omega(\rn)}$
(see, for instance, \cite[Remark 1.4.7]{g14}).
\end{remark}

In what follows, for any $x\in\rn$ and $r\in(0,\fz)$, we always let
$B(x,r):=\{y\in\rn:\ |y-x|<r\}$.

\begin{definition}\label{d1.5}
Let $q\in(0,\fz)$, $t\in(0,\fz]$, and $\tz\in[0,n]$.
The \emph{Lorentz--Morrey space} $L^{q,\,t;\,\tz}(\rn)$ is defined by setting
$$L^{q,\,t;\,\tz}(\rn):=\lf\{f\ \text{is measurable on}\ \rn:\
\|f\|_{L^{q,\,t;\,\tz}(\rn)}<\fz\r\},
$$
where
$$\|f\|_{L^{q,\,t;\,\tz}(\rn)}:=\sup_{x\in\rn,\,r\in(0,\fz)}
\lf\{r^{\frac{\tz-n}{q}}\|f\|_{L^{q,\,t}(B(x,r))}\r\}:=\sup_{x\in\rn,\,r\in(0,\fz)}
\lf\{r^{\frac{\tz-n}{q}}\|f\mathbf{1}_{B(x,r)}\|_{L^{q,\,t}(\rn)}\r\}
$$
and  $\mathbf{1}_{B(x,r)}$ denotes the characteristic function of $B(x,r)$.
\end{definition}

We point out that, when $\tz=n$, the Lorentz--Morrey space $L^{q,\,t;\,\tz}(\rn)$
is just the \emph{Lorentz space}; in this case, we denote the space $L^{q,\,t;\,\tz}(\rn)$
simply by $L^{q,\,t}(\rn)$. Moreover, when $q=t$, the space $L^{q,\,t;\,\tz}(\rn)$
is just the \emph{Morrey space}; in this case, we denote the space $L^{q,\,t;\,\tz}(\rn)$
simply by $\cm^{\tz}_q(\rn)$ (see, for instance, \cite{a15,sdh20} for more details on Morrey spaces).

\begin{definition}\label{d1.6}
Let $\mathcal{P}(\rn)$ be the set of all measurable functions $p:\,\rn\to(0,\fz)$.
For any given $p(\cdot)\in\mathcal{P}(\rn)$,
the \emph{variable exponent modular} $\rho_{p(\cdot)}$ is defined by setting,
for any $f\in L^1_\loc(\rn)$,
$$\rho_{p(\cdot)}(f):=\int_{\rn}|f(x)|^{p(x)}\,dx.$$
The \emph{variable Lebesgue space} $L^{p(\cdot)}(\rn)$ is defined by setting
\begin{align*}
L^{p(\cdot)}(\rn)=\lf\{f\ \text{is measurable on}\ \rn:\  \text{there exists a}\ \lz\in(0,\fz)\
\text{such that}\ \rho_{p(\cdot)}(\lz f)<\fz\r\}
\end{align*}
equipped with the \emph{Luxembourg norm} (which is also called the
\emph{Luxembourg--Nakano norm})
$$\|f\|_{L^{p(\cdot)}(\rn)}:=\inf\lf\{\lz\in(0,\fz):\
\rho_{p(\cdot)}\lf(\frac{f}{\lz}\r)\le1\r\}
$$
(see, for instance, \cite{cf13,dhhr11} for more details on variable Lebesgue spaces).
Moreover, for any $p(\cdot)\in\mathcal{P}(\rn)$, let
\begin{equation}\label{1.7}
p_-:=\mathop{\mathrm{ess\,inf}}\limits_{x\in\rn}p(x)\ \quad
\text{and} \quad \ p_+:=\mathop{\mathrm{ess\,sup}}\limits_{x\in\rn}p(x).
\end{equation}
\end{definition}

\begin{remark}\label{r1.3}
Let $p(\cdot)\in\cp(\rn)$ satisfy $p_+<\fz$. Then, for any $s\in(0,\fz)$ and $f\in L^{sp(\cdot)}(\rn)$,
$\||f|^s\|_{L^{p(\cdot)}(\rn)}=\|f\|^s_{L^{sp(\cdot)}(\rn)}$
(see, for instance, \cite[Proposition 2.18]{cf13}).
\end{remark}

Applying the weighted good-$\lz$ inequality given in Theorem \ref{t1.1}, and the two-weight boundedness of
the Hardy--Littlewood maximal operator $\cm$, we obtain the following (two-weight) boundedness criterion
for a pair of functions, $(F,f)$, on $\rn$ satisfying the assumptions \eqref{1.4} and \eqref{1.5} in the scale of
weighted Lebesgue spaces, weighted Lorentz spaces, (Lorentz--)Morrey spaces, and variable Lebesgue spaces.

\begin{theorem}\label{t1.2}
Let $p_1,\,p_2,\,p_3\in(0,\fz]$ satisfy $p_3>\max\{p_1,\,p_2\}$, $q\in(\max\{p_1,\,p_2\},p_3)$, $\Phi$ be a doubling
Young function satisfying
$$\int_c^\fz\lf[\frac{t^{(\frac{q}{p_2})'}}{\Phi(t)}\r]^{\frac{q}{p_2}-1}\frac{dt}{t}<\fz
$$
for some constant $c\in(0,\fz)$,
and $F,\,f\in L^{1}_{\loc}(\rn)$. Assume that the weights $\omega$ and $v$ satisfy
that $\omega\in RH_s(\rn)$ with some $s\in((\frac{p_3}{q})',\fz]$, and
\begin{equation}\label{1.8}
\sup_{B\subset\rn}\lf(\fint_B \omega\,dx\r)\lf\|v^{-\frac{p_2}{q}}\r\|^{\frac{q}{p_2}}_{\Phi,\,B}<\fz,
\end{equation}
where the supremum is taken over all balls $B$ in $\rn$.
Assume further that $F$ and $f$ satisfy \eqref{1.4} and \eqref{1.5} with $\gz=0$ and some $\epsilon\in(0,\fz)$
such that \eqref{1.6} holds true.
\begin{itemize}
\item[\rm(i)] Then there exists a positive constant $C$, independent of $F$ and $f$, such that
\begin{equation*}
\|F\|_{L^q_\omega(\rn)}\le C\|f\|_{L^q_v(\rn)}.
\end{equation*}

\item[\rm(ii)] Let $t\in(0,\fz]$. Assume further that there exists a $q_0\in(1,\frac{q}{p_2})$ such that
$$\int_c^\fz\lf[\frac{t^{q_0'}}{\Phi(t)}\r]^{q_0-1}\frac{dt}{t}<\fz
$$
for some constant $c\in(0,\fz)$, and the weights $\omega$ and $v$ satisfy
\begin{equation*}
\sup_{B\subset\rn}\lf(\fint_B \omega\,dx\r)\lf\|v^{-\frac{1}{q_0}}\r\|^{q_0}_{\Phi,\,B}<\fz
\end{equation*}
with the supremum taken over all balls $B$ of $\rn$.
Then there exists a positive constant $C$, independent of $F$ and $f$, such that
\begin{equation*}
\|F\|_{L^{q,\,t}_\omega(\rn)}\le C\|f\|_{L^{q,\,t}_v(\rn)}.
\end{equation*}

\item[\rm(iii)] Let $t\in(0,\fz]$ and $\tz\in(\frac{nq}{p_3},n]$.
Then there exists a positive constant $C$, independent of $F$ and $f$, such that
\begin{equation*}
\|F\|_{L^{q,\,t;\,\tz}(\rn)}\le C\|f\|_{L^{q,\,t;\,\tz}(\rn)}.
\end{equation*}

In particular, there exists a positive constant $C$, independent of $F$ and $f$,
such that
\begin{equation*}
\|F\|_{\cm^{\tz}_q(\rn)}\le C\|f\|_{\cm^{\tz}_q(\rn)}.
\end{equation*}

\item[\rm(iv)] Let $p(\cdot)\in\mathcal{P}(\rn)$, and $p_-$ and $p_+$ be as in \eqref{1.7}.
Assume that $p_2<p_-\le p_+< p_3$ and $\cm$ is bounded on $L^{\frac{p(\cdot)}{p_2}}(\rn)$.
Then there exists a positive constant $C$, independent of $F$ and $f$, such that
\begin{equation*}
\|F\|_{L^{p(\cdot)}(\rn)}\le C\|f\|_{L^{p(\cdot)}(\rn)}.
\end{equation*}
\end{itemize}
\end{theorem}

A \emph{distinctive feature} of the (two-weight) boundedness criterion obtained in Theorem \ref{t1.2} is
its independence of specific operators. It is flexible enough to obtain the boundedness
of both linear operators and non-linear operators in the scale of weighted Lebesgue spaces,
weighted Lorentz spaces, (Lorentz--)Morrey spaces, and variable Lebesgue spaces.

We prove (i) and (ii) of Theorem \ref{t1.2} by using the weighted good-$\lz$ inequality established
in Theorem \ref{t1.1} and the two-weight boundedness of the Hardy--Littlewood maximal
operator $\cm$ in the scale of both weighted Lebesgue spaces and weighted Lorentz spaces.
Then, via choosing special weights $\omega$ and $v$, and applying the conclusion of
Theorem \ref{t1.2}(ii), we show (iii) of Theorem \ref{t1.2}.
We also point out that, for the proof of Theorem \ref{t1.2}(iii),
we borrow some ideas from the proof of \cite[Theorem 2.3]{mp12}.
Finally, using the special case $\omega=v$
of Theorem \ref{t1.2}(i) and the limited range extrapolation theorem in the
variable exponent case established by Cruz-Uribe and Wang \cite[Theorem 2.14]{cw17}
(see also Lemma \ref{l3.7} below), we prove (iv) of Theorem \ref{t1.2}.

Moreover, applying Theorem \ref{t1.1} and the two-weight boundedness of the fractional Hardy--Littlewood maximal operator
$\cm_\gz$, we obtain the following off-diagonal (two-weight) boundedness criterion for a pair of functions, $(F,f)$,
on $\rn$ satisfying the assumptions \eqref{1.4} and \eqref{1.5} in the scale of both weighted Lebesgue spaces and Morrey spaces.

\begin{theorem}\label{t1.3}
Let $\gz\in(0,1)$, $p_1,\,p_2,\,p_3\in(0,\fz]$ satisfy
$$p_3>\max\{p_1,\,p_2\}\quad\text{and}\quad  q\in\lf(\max\{p_1,\,p_2\},p_3\r),$$
and $F,\,f\in L^{1}_{\loc}(\rn)$.
Assume that $q(1-\gz)>p_2$, $p\in(p_2,\fz)$ is given by setting
$\frac{1}{p}:=\frac{1}{q}+\frac{\gz}{p_2}$, and $\Phi$ is a doubling Young function satisfying
$$\int_c^\fz\lf[\frac{t^{(\frac{p}{p_2})'}}{\Phi(t)}\r]^{\frac{p}{p_2}-1}\frac{dt}{t}<\fz
$$
for some constant $c\in(0,\fz)$.  Let the weights $\omega$ and $v$ satisfy
that $\omega\in RH_s(\rn)$ with some $s\in((\frac{p_3}{q})',\fz]$, and
\begin{equation}\label{1.9}
\sup_{B\subset\rn}\lf(\fint_B \omega\,dx\r)\lf\|v^{-\frac{p_2}{p}}\r\|^{\frac{q}{p_2}}_{\Phi,\,B}<\fz,
\end{equation}
where the supremum is taken over all balls $B$ of $\rn$.
Assume further that $F$ and $f$ satisfy \eqref{1.4} and \eqref{1.5}
with some $\epsilon\in(0,\fz)$ such that \eqref{1.6} holds true.
\begin{itemize}
\item[\rm(i)] Then there exists a positive constant $C$, independent of $F$ and $f$, such that
\begin{equation*}
\|F\|_{L^q_\omega(\rn)}\le C\|f\|_{L^{p}_v(\rn)}.
\end{equation*}

\item[\rm(ii)] Let $\tz\in(\frac{nq}{p_3},n]$ and $\wz{\tz}:=n-\frac{p}{q}(n-\tz)$.
Then there exists a positive constant $C$, independent of $F$ and $f$, such that
\begin{equation*}
\|F\|_{\cm^\tz_q(\rn)}\le C\|f\|_{\cm^{\wz{\tz}}_p(\rn)}.
\end{equation*}
\end{itemize}
\end{theorem}

We show (i) of Theorem \ref{t1.3} by using the weighted good-$\lz$ inequality given
in Theorem \ref{t1.1} and the two-weight boundedness of the fractional Hardy--Littlewood maximal operator
$\cm_\gz$ in the scale of weighted Lebesgue spaces. Then, choosing special weights $\omega$ and $v$ and
applying Theorem \ref{t1.3}(i), we prove (ii) of Theorem \ref{t1.3}.

\begin{remark}\label{r1.4}
The two-weight boundedness of the fractional Hardy--Littlewood maximal operator
$\cm_\gz$ from the weighted Lorentz space $L^{p,\,s}_v(\rn)$ to $L^{q,\,t}_\omega(\rn)$ was studied by Rakotondratsimba \cite{r97}
under some assumptions for the weights $\omega$ and $v$, and the exponents $p,\,q,\,s$, and $t$
(see \cite[Corollary 2.7]{r97} for the details). Using the two-weight boundedness of
$\cm_\gz$ in the scale of weighted Lorentz spaces obtained by Rakotondratsimba \cite[Corollary 2.7]{r97},
and Theorem \ref{t1.1}, we can obtain the off-diagonal two-weight boundedness criterion for a pair of functions, $(F,f)$,
on $\rn$ in the scale of both weighted Lorentz spaces and Lorentz--Morrey spaces,
under some additional assumptions for the weights $\omega$ and $v$,
and the exponents appeared in the considered weighted Lorentz spaces and Lorentz--Morrey spaces.
Moreover, under the additional assumption $p_3=\fz$ in \eqref{1.4}, applying the conclusion of Theorem \ref{t1.3}(i)
and the off-diagonal extrapolation theorem in the scale of variable Lebesgue spaces obtained
by Cruz-Uribe and Wang \cite[Theorem 2.11]{cw17}, we can obtain the off-diagonal boundedness
criterion for a pair of functions, $(F,f)$, on $\rn$ in the case of variable Lebesgue spaces.
In order to limit the length of this article, we omit the details here.
\end{remark}

As applications of the general (two-weight) boundedness criterion obtained in Theorems \ref{t1.2}
and \ref{t1.3}, we prove the (two-weight) boundedness of Calder\'on--Zygmund operators, Littlewood--Paley $g$-functions,
Lusin area functions, Littlewood--Paley $g^\ast_\lz$-functions, and fractional integral operators
in the scale of weighted Lebesgue spaces, weighted Lorentz spaces, (Lorentz--)Morrey spaces,
and variable Lebesgue spaces, which have independent interests
and are presented in Subsection \ref{s2.1} below. Although some of these boundedness are known,
the proofs given in this article are quite different from those used in the existing literatures.
Indeed, using the boundedness of Calder\'on--Zygmund operators, Littlewood--Paley $g$-functions,
and Lusin area functions from $L^1(\rn)$ to $L^{1,\,\fz}(\rn)$, and the boundedness of
fractional integral operators $I_\az$, with $\az\in(0,n)$, from $L^1(\rn)$ to
$L^{\frac{n}{n-\az},\,\fz}(\rn)$, and applying Theorems \ref{t1.2} and \ref{t1.3}, we
obtain the above boundedness for Calder\'on--Zygmund operators, Littlewood--Paley $g$-functions,
Lusin area functions, Littlewood--Paley $g^\ast_\lz$-functions, and fractional integral operators.

Moreover, using the (two-weight) boundedness criterion given in Theorems \ref{t1.2}
and \ref{t1.3}, we obtain the (two-weight) boundedness of Riesz transforms, Littlewood--Paley $g$-functions,
and fractional integral operators associated with second-order divergence elliptic operators with
complex bounded measurable coefficients on $\rn$ in the scale of weighted Lebesgue spaces,
weighted Lorentz spaces, (Lorentz--)Morrey spaces,
and variable Lebesgue spaces, which are presented in Subsection \ref{s2.2} below.

The organization of the remainder of
of this article is as follows. Applications of Theorems \ref{t1.2} and \ref{t1.3}
are presented in Section \ref{s2} and their proofs are given, respectively, in Sections \ref{s4}
and \ref{s5}, while, in Section \ref{s3}, we show Theorems \ref{t1.1}, \ref{t1.2},
and \ref{t1.3}.

Finally, we make some conventions on notation.
Throughout this article, we always denote by $C$ a \emph{positive constant} which is
independent of the main parameters, but it may vary from line to
line. We also use $C_{(\gz,\,\bz,\,\ldots)}$ or $c_{(\gz,\,\bz,\,\ldots)}$
to denote a  \emph{positive constant} depending on the indicated parameters $\gz,$ $\bz$,
$\ldots$. The \emph{symbol} $f\ls g$ means that $f\le Cg$. If $f\ls
g$ and $g\ls f$, then we write $f\sim g$. If $f\le Cg$ and $g=h$ or $g\le h$, we then write $f\ls g\sim h$
or $f\ls g\ls h$, \emph{rather than} $f\ls g=h$ or $f\ls g\le h$.
For each ball $B:=B(x_B,r_B)$ in $\rn$, with some $x_B\in\rn$ and
$r_B\in (0,\fz)$, and $\az\in(0,\fz)$, let $\az B:=B(x_B,\az r_B)$.
For any measurable subset $E$ of $\rn$, we denote the \emph{set} $\rn\setminus E$ by $E^\complement$,
and its \emph{characteristic function} by $\mathbf{1}_{E}$.
For any weight $\omega$ on $\rn$ and any measurable set $E$ of $\rn$,
let $\omega(E):=\int_E\omega(x)\,dx$. For any given $q\in[1,\fz]$, we denote by $q'$
its \emph{conjugate exponent}, namely, $1/q + 1/q'= 1$.
Finally, for any measurable set $E$ of $\rn$, a weight $\omega$ on $\rn$,
and $f\in L^1(E)$, we denote the integral $\int_E|f(x)|\omega(x)\,dx$
simply by $\int_E|f|\omega\,dx$ and, when $|E|\in(0,\fz)$, we let
$$\fint_E f\,dx:=\frac{1}{|E|}\int_Ef(x)\,dx.$$

\section{Applications of Theorems \ref{t1.2} and \ref{t1.3}\label{s2}}

In this section, we give some applications of the general (two-weight) boundedness criterion
obtained in Theorems \ref{t1.2} and \ref{t1.3}, whose proofs are given, respectively,
in Sections \ref{s4} and \ref{s5}. More precisely, using Theorems \ref{t1.2} and \ref{t1.3},
we prove the (two-weight) boundedness of Calder\'on--Zygmund operators, Littlewood--Paley $g$-functions,
Lusin area functions, Littlewood--Paley $g^\ast_\lz$-functions, and fractional integral operators
in the scale of weighted Lebesgue spaces, weighted Lorentz spaces, (Lorentz--)Morrey spaces,
and variable Lebesgue spaces. Moreover, via using Theorems \ref{t1.2} and \ref{t1.3},
we obtain the (two-weight) boundedness of Riesz transforms, Littlewood--Paley $g$-functions,
and fractional integral operators associated with second-order divergence elliptic operators with
complex bounded measurable coefficients on $\rn$ in the scale of weighted Lebesgue spaces,
weighted Lorentz spaces, (Lorentz--)Morrey spaces, and variable Lebesgue spaces.

\subsection{Applications to some classical operators}\label{s2.1}

In this subsection, we present the (two-weight) boundedness of classical Calder\'on--Zygmund operators,
Littlewood--Paley $g$-functions, Lusin area functions, Littlewood--Paley $g^\ast_\lz$-functions,
and fractional integral operators in the scale of weighted Lebesgue spaces, weighted Lorentz spaces,
(Lorentz--)Morrey spaces, and variable Lebesgue spaces. We begin with the definitions of Calder\'on--Zygmund operators.

\begin{definition}\label{d2.1}
Let $\dz\in(0,1]$. A linear operator $T$ is called a \emph{Calder\'on--Zygmund operator} if
$T$ is bounded on $L^2(\rn)$ and, for any $f\in L^2(\rn)$ having compact support
and $x\in \supp(f):=\{y\in\rn:\ f(y)\neq0\}$,
$$T(f)(x):=\int_{\rn}K(x,y)f(y)\,dy,
$$
where the kernel function $K$ satisfies that there exists a positive constant
$C$ such that, for any $x,\,y,\,h\in\rn$ with $|h|<|x-y|/2$,
$$|K(x,y)|\le\frac{C}{|x-y|^n},
$$
$$|K(x+h,y)-K(x,y)|\le\frac{C|h|^\dz}{|x-y|^{n+\dz}},\quad\text{and}\quad |K(y,x+h)-K(y,x)|\le\frac{C|h|^\dz}{|x-y|^{n+\dz}}.
$$
\end{definition}

Then we have the following (two-weight) boundedness for Calder\'on--Zygmund operators.

\begin{theorem}\label{t2.1}
Let $q\in(1,\fz)$ and $\Phi$ be a doubling Young function satisfying
$$\int_c^\fz\lf[\frac{t^{q'}}{\Phi(t)}\r]^{q-1}\frac{dt}{t}<\fz
$$
for some constant $c\in(0,\fz)$. Assume that the weights $\omega$ and $v$ satisfy
that $\omega\in A_\fz(\rn)$ and
$$\sup_{B\subset\rn}\lf(\fint_B \omega\,dx\r)\lf\|v^{-\frac{1}{q}}\r\|^{q}_{\Phi,\,B}<\fz,
$$
where the supremum is taken over all balls $B$ of $\rn$.
Let $T$ be a Calder\'on--Zygmund operator.
\begin{itemize}
\item[\rm(i)] Then $T$ is bounded from $L^q_v(\rn)$ to $L^q_\omega(\rn)$, and
there exists a positive constant $C$ such that, for any $f\in L^q_v(\rn)$,
\begin{equation*}
\|T(f)\|_{L^q_\omega(\rn)}\le C\|f\|_{L^q_v(\rn)}.
\end{equation*}

\item[\rm(ii)] Let $t\in(0,\fz]$. Assume further that there exists a $q_0\in(1,q)$ such that
$$\int_c^\fz\lf[\frac{t^{q_0'}}{\Phi(t)}\r]^{q_0-1}\frac{dt}{t}<\fz
$$
for some constant $c\in(0,\fz)$, and the weights $\omega$ and $v$ satisfy
$$\sup_{B\subset\rn}\lf(\fint_B \omega\,dx\r)\lf\|v^{-\frac{1}{q_0}}\r\|^{q_0}_{\Phi,\,B}<\fz
$$
with the supremum taken over all balls $B$ of $\rn$. Then $T$ is bounded
from $L^{q,\,t}_v(\rn)$ to $L^{q,\,t}_\omega(\rn)$,
and there exists a positive constant $C$ such that, for any $f\in L^{q,\,t}_v(\rn)$,
\begin{equation*}
\|T(f)\|_{L^{q,\,t}_\omega(\rn)}\le C\|f\|_{L^{q,\,t}_v(\rn)}.
\end{equation*}

\item[\rm(iii)] Let $t\in(0,\fz]$ and $\tz\in(0,n]$.
Then $T$ is bounded on $L^{q,\,t;\,\tz}(\rn)$,
and there exists a positive constant $C$ such that, for any $f\in L^{q,\,t;\,\tz}(\rn)$,
\begin{equation*}
\|T(f)\|_{L^{q,\,t;\,\tz}(\rn)}\le C\|f\|_{L^{q,\,t;\,\tz}(\rn)}.
\end{equation*}

In particular, $T$ is bounded on $\cm^{\tz}_q(\rn)$,
and there exists a positive constant $C$ such that, for any $f\in \cm^{\tz}_q(\rn)$,
\begin{equation*}
\|T(f)\|_{\cm^{\tz}_q(\rn)}\le C\|f\|_{\cm^{\tz}_q(\rn)}.
\end{equation*}

\item[\rm(iv)] Assume that $p(\cdot)\in\mathcal{P}(\rn)$ satisfies $1<p_-\le p_+< \fz$ with $p_-$ and $p_+$
as in \eqref{1.7}, and that $\cm$ is bounded on $L^{p(\cdot)}(\rn)$.
Then $T$ is bounded on $L^{p(\cdot)}(\rn)$,
and there exists a positive constant $C$ such that, for any $f\in L^{p(\cdot)}(\rn)$,
\begin{equation*}
\|T(f)\|_{L^{p(\cdot)}(\rn)}\le C\|f\|_{L^{p(\cdot)}(\rn)}.
\end{equation*}
\end{itemize}
\end{theorem}

Denote by $L^\fz_{\rm c}(\rn)$ the set of all bounded and measurable functions
on $\rn$ with compact support. Let $f\in L^\fz_{\rm c}(\rn)$, $F:=T(f)$, and $\nu\in(0,1)$ be a constant.
For any ball $B$ of $\rn$,  let $F_B:=T(f\mathbf{1}_{8B})$ and $R_B:=T(f-f\mathbf{1}_{8B})$.
To prove Theorem \ref{t2.1} via applying Theorem \ref{t1.2}, it suffices to show that \eqref{1.4} and \eqref{1.5} hold true
for $p_1:=\nu$, $p_2:=1$, $p_3:=\fz$, and $\epsilon:=0$. We prove this by using the boundedness of $T$
from $L^1(\rn)$ to $L^{1,\,\fz}(\rn)$, and the Kolmogorov inequality (see Section \ref{s4} below for the details).

\begin{remark}\label{r2.1}
\begin{itemize}
\item[\rm(i)] The conclusion of Theorem \ref{t2.1}(i) is well known (see, for instance, \cite{cmp07,crv14,cp02,r98a}).
It is worth pointing out that the proof of Theorem \ref{t2.1}(i) presented in this article is quite different
from that given in \cite{cmp07,crv14,cp02,r98a}. More precisely, Theorem \ref{t2.1}(i) was proved
in \cite{cmp07,crv14,cp02,r98a} by using the sharp maximal function control for the operator $T$ or the Haar shift theory.
However, Theorem \ref{t2.1}(i) is showed in this article via applying the two-weight boundedness criterion
obtained in Theorem \ref{t1.2}.
\item[\rm(ii)] Theorem \ref{t2.1}(ii) was established in \cite{km97,r98c} under a different assumption
on the weights $\omega$ and $v$.
\item[\rm(iii)] The boundedness of the Calder\'on--Zygmund operator $T$ on the Morrey space $\cm^\tz_q(\rn)$
is well known (see, for instance, \cite[Theorem 8.1]{a15}). To the best of our knowledge, the boundedness of
$T$ on the Lorentz--Morrey space obtained in Theorem \ref{t2.1}(iii) is new.
\item[\rm(iv)] Theorem \ref{t2.1}(vi) is known (see, for instance, \cite[Theorem 4.8]{dr03}).
However, the method used in this article for the proof of Theorem \ref{t2.1}(vi) is different from
that used in \cite[Theorem 4.8]{dr03}. Indeed, Theorem \ref{t2.1}(vi) was showed in \cite[Theorem 4.8]{dr03}
by using the sharp maximal function control for the operator $T$ and the boundedness of $\cm$ on $L^{p(\cdot)}(\rn)$.
\end{itemize}
\end{remark}

Denote by $\mathcal{S}(\rn)$ the \emph{Schwartz space} equipped with the well-known classical
topology determined by a countable family of norms. Furthermore, denote by $\mathcal{S}'(\rn)$ the topological dual
space of $\mathcal{S}(\rn)$ equipped the weak-$\ast$ topology.
For any $\varphi\in\mathcal{S}(\rn)$ and $t\in(0,\infty)$, let $\varphi_t(\cdot):= \frac{1}{t^n}\varphi(\frac{\cdot}{t})$.

\begin{definition}\label{d2.2}
Let $\phi\in\mathcal{S}(\rn)$ satisfy $\int_{\rn}\phi\,dx=0$, and $f\in\mathcal{S}'(\rn)$.
The \emph{Littlewood--Paley $g$-function} $g(f)$ is defined by setting, for any $x\in\rn$,
$$g(f)(x):=\lf[\int_0^\fz\lf|\phi_t\ast f(x)\r|^2\frac{dt}{t}\r]^{\frac12}.
$$
Let $\az\in[1,\fz)$. The \emph{Lusin area function $S_\az(f)$} is defined by setting, for any $x\in\rn$,
$$S_\az(f)(x):=\lf[\int_0^\fz\int_{\Gamma_\az(x)}|\phi_t\ast f(x)|^2\frac{dy\,dt}{t^{n+1}}\r]^{\frac12},
$$
where $\Gamma_\az(x):=\{y\in\rn:\ |x-y|<\az t\}$. In particular, when $\az:=1$,
$S_\az(f)$ is denoted simply by $S(f)$.

Let $\lz\in(0,\fz)$. The \emph{Littlewood--Paley $g^\ast_\lz$-function}
$g^\ast_\lz(f)$ is defined by setting, for any $x\in\rn$,
$$g^\ast_\lz(f)(x):=\lf[\int_0^\fz\int_{\rn}\lf(\frac{t}{t+|y|}\r)^{\lz n}|\phi_t\ast
f(x-y)|^2\frac{dy\,dt}{t^{n+1}}\r]^{\frac12}.
$$
\end{definition}

We then have the following (two-weight) boundedness of Littlewood--Paley $g$-functions.

\begin{theorem}\label{t2.2}
Let $q\in(1,\fz)$, $\Phi$, $\omega$, and $v$ be as in Theorem \ref{t2.1}.

\begin{itemize}
\item[\rm(i)] Then there exists a positive constant $C$ such that, for any $f\in L^q_v(\rn)$,
\begin{equation*}
\lf\|g(f)\r\|_{L^q_\omega(\rn)}\le C\|f\|_{L^q_v(\rn)}.
\end{equation*}

\item[\rm(ii)] Let $t\in(0,\fz]$. Assume further that there exists a $q_0\in(1,q)$ such that
$$\int_c^\fz\lf[\frac{t^{q_0'}}{\Phi(t)}\r]^{q_0-1}\frac{dt}{t}<\fz
$$
for some constant $c\in(0,\fz)$, and the weights $\omega$ and $v$ satisfy
$$\sup_{B\subset\rn}\lf(\fint_B \omega\,dx\r)\lf\|v^{-\frac{1}{q_0}}\r\|^{q_0}_{\Phi,\,B}<\fz
$$
with the supremum taken over all balls $B$ of $\rn$.
Then there exists a positive constant $C$ such that, for any $f\in L^{q,\,t}_v(\rn)$,
\begin{equation*}
\|g(f)\|_{L^{q,\,t}_\omega(\rn)}\le C\|f\|_{L^{q,\,t}_v(\rn)}.
\end{equation*}

\item[\rm(iii)] Let $t\in(0,\fz]$ and $\tz\in(0,n]$.
Then there exists a positive constant $C$ such that, for any $f\in L^{q,\,t;\,\tz}(\rn)$,
\begin{equation*}
\|g(f)\|_{L^{q,\,t;\,\tz}(\rn)}\le C\|f\|_{L^{q,\,t;\,\tz}(\rn)}.
\end{equation*}

In particular, there exists a positive constant $C$ such that,
for any $f\in \cm^\tz_q(\rn)$,
$$\|g(f)\|_{\cm^\tz_q(\rn)}\le C\|f\|_{\cm^\tz_q(\rn)}.$$

\item[\rm(iv)] Assume that $p(\cdot)\in\mathcal{P}(\rn)$ satisfies $1<p_-\le p_+< \fz$ with
$p_-$ and $p_+$ as in \eqref{1.7}, and that $\cm$ is bounded on $L^{p(\cdot)}(\rn)$.
Then there exists a positive constant $C$ such that, for any $f\in L^{p(\cdot)}(\rn)$,
\begin{equation*}
\|g(f)\|_{L^{p(\cdot)}(\rn)}\le C\|f\|_{L^{p(\cdot)}(\rn)}.
\end{equation*}
\end{itemize}
\end{theorem}

\begin{remark}\label{r2.2}
We point out that the two-weight boundedness of the Littlewood--Paley $g$-function
in the scale of weighted Lebesgue spaces was established in \cite{ll14} under a different
assumption on the weights $\omega$ and $v$. To the best of our knowledge, the boundedness of
the Littlewood--Paley $g$-function on weighted Lorentz spaces and Lorentz--Morrey spaces obtained in
(ii) and (iii) of Theorem \ref{t2.2} is new.

The boundedness of the Littlewood--Paley $g$-function on $L^{p(\cdot)}(\rn)$ is a direct corollary of
the one-weight boundedness of the Littlewood--Paley $g$-function in the scale of weighted Lebesgue spaces
(see, for instance, \cite{w08}) and the Rubio de Francia extrapolation theorem (see, for instance, \cite{cmp11}).
Thus, the conclusion of Theorem \ref{t2.2}(iv) is known.
\end{remark}

For the Lusin area function, we have the following (two-weight) boundedness.

\begin{theorem}\label{t2.3}
Let $q\in(1,\fz)$, $\az\in[1,\fz)$, $\Phi$, $\omega$, and $v$ be as in Theorem \ref{t2.1}.

\begin{itemize}
\item[\rm(i)] Then there exists a positive constant $C$, independent of $\az$, such that, for any $f\in L^q_v(\rn)$,
\begin{equation*}
\lf\|S_\az(f)\r\|_{L^q_\omega(\rn)}\le C\az^n\|f\|_{L^q_v(\rn)}.
\end{equation*}

\item[\rm(ii)] Let $t\in(0,\fz]$. Assume further that there exists a $q_0\in(1,q)$ such that
$$\int_c^\fz\lf[\frac{t^{q_0'}}{\Phi(t)}\r]^{q_0-1}\frac{dt}{t}<\fz
$$
for some constant $c\in(0,\fz)$, and the weights $\omega$ and $v$ satisfy
$$\sup_{B\subset\rn}\lf(\fint_B \omega\,dx\r)\lf\|v^{-\frac{1}{q_0}}\r\|^{q_0}_{\Phi,\,B}<\fz
$$
with the supremum taken over all balls $B$ of $\rn$.
Then there exists a positive constant $C$, independent of $\az$, such that,
for any $f\in L^{q,\,t}_v(\rn)$,
\begin{equation*}
\lf\|S_\az(f)\r\|_{L^{q,\,t}_\omega(\rn)}\le C\az^n\|f\|_{L^{q,\,t}_v(\rn)}.
\end{equation*}

\item[\rm(iii)] Let $t\in(0,\fz]$ and $\tz\in(0,n]$.
Then there exists a positive constant $C$, independent of $\az$, such that, for any $f\in L^{q,\,t;\,\tz}(\rn)$,
\begin{equation*}
\lf\|S_\az(f)\r\|_{L^{q,\,t;\,\tz}(\rn)}\le C\az^n\|f\|_{L^{q,\,t;\,\tz}(\rn)}.
\end{equation*}

In particular, there exists a positive constant $C$, independent of $\az$, such that,
for any $f\in \cm^\tz_q(\rn)$,
$$\lf\|S_\az(f)\r\|_{\cm^\tz_q(\rn)}\le C\az^n\|f\|_{\cm^\tz_q(\rn)}.$$

\item[\rm(iv)] Assume that $p(\cdot)\in\mathcal{P}(\rn)$ satisfies $1<p_-\le p_+< \fz$ with
$p_-$ and $p_+$ as in \eqref{1.7}, and that $\cm$ is bounded on $L^{p(\cdot)}(\rn)$.
Then there exists a positive constant $C$, independent of $\az$, such that, for any $f\in L^{p(\cdot)}(\rn)$,
\begin{equation*}
\lf\|S_\az(f)\r\|_{L^{p(\cdot)}(\rn)}\le C\az^n\|f\|_{L^{p(\cdot)}(\rn)}.
\end{equation*}
\end{itemize}
\end{theorem}

The proofs of Theorems \ref{t2.2} and \ref{t2.3} are similar to that of Theorem \ref{t2.1}.
Precisely, using the boundedness of the Littlewood-Paley $g$-function and the Lusin area function
from $L^1(\rn)$ to $L^{1,\,\fz}(\rn)$, and the Kolmogorov inequality, and applying Theorem \ref{t1.2},
we show Theorems \ref{t2.2} and \ref{t2.3} in Section \ref{s4}.

Applying Theorem \ref{t2.3} and the poinwise estimate
\begin{equation}\label{2.1}
g^\ast_\lz(f)\le S(f)+\sum_{k=0}^\fz2^{-k\lz n/2}S_{2^{k+1}}(f)
\end{equation}
(see, for instance, \cite[p.\,786]{l14}),
we obtain the following (two-weight) boundedness estimates for the Littlewood--Paley $g^\ast_\lz$-function.

\begin{corollary}\label{c2.1}
Let $q\in(1,\fz)$, $\lz\in(2,\fz)$, $\Phi$, $\omega$, and $v$ be as in Theorem \ref{t2.1}.

\begin{itemize}
\item[\rm(i)] Then there exists a positive constant $C$ such that, for any $f\in L^q_v(\rn)$,
\begin{equation*}
\lf\|g^\ast_\lz(f)\r\|_{L^q_\omega(\rn)}\le C\|f\|_{L^q_v(\rn)}.
\end{equation*}

\item[\rm(ii)] Let $t\in(0,\fz]$. Assume further that there exists a $q_0\in(1,q)$ such that
$$\int_c^\fz\lf[\frac{t^{q_0'}}{\Phi(t)}\r]^{q_0-1}\frac{dt}{t}<\fz
$$
for some constant $c\in(0,\fz)$, and the weights $\omega$ and $v$ satisfy
\begin{equation}\label{2.2}
\sup_{B\subset\rn}\lf(\fint_B \omega\,dx\r)\lf\|v^{-\frac{1}{q_0}}\r\|^{q_0}_{\Phi,\,B}<\fz
\end{equation}
with the supremum taken over all balls $B$ of $\rn$. Then there exists a positive constant $C$ such that,
for any $f\in L^{q,\,t}_v(\rn)$,
\begin{equation*}
\lf\|g^\ast_\lz(f)\r\|_{L^{q,\,t}_\omega(\rn)}\le C\|f\|_{L^{q,\,t}_v(\rn)}.
\end{equation*}

\item[\rm(iii)] Let $t\in(0,\fz]$ and $\tz\in(0,n]$.
Then there exists a positive constant $C$ such that, for any $f\in L^{q,\,t;\,\tz}(\rn)$,
\begin{equation*}
\lf\|g^\ast_\lz(f)\r\|_{L^{q,\,t;\,\tz}(\rn)}\le C\|f\|_{L^{q,\,t;\,\tz}(\rn)}.
\end{equation*}

In particular, there exists a positive constant $C$ such that,
for any $f\in \cm^\tz_q(\rn)$,
$$\lf\|g^\ast_\lz(f)\r\|_{\cm^\tz_q(\rn)}\le C\|f\|_{\cm^\tz_q(\rn)}.$$

\item[\rm(iv)] Assume that $p(\cdot)\in\mathcal{P}(\rn)$ satisfies $1<p_-\le p_+< \fz$
with $p_-$ and $p_+$ as in \eqref{1.7}, and that $\cm$ is bounded on $L^{p(\cdot)}(\rn)$.
Then there exists a positive constant $C$ such that, for any $f\in L^{p(\cdot)}(\rn)$,
\begin{equation*}
\lf\|g^\ast_\lz(f)\r\|_{L^{p(\cdot)}(\rn)}\le C\|f\|_{L^{p(\cdot)}(\rn)}.
\end{equation*}
\end{itemize}
\end{corollary}

By Theorem \ref{t2.3}(i), \eqref{2.1}, and the Minkowski inequality
for the space $L^q_v(\rn)$, we immediately obtain Corollary \ref{c2.1}(i).
Moreover, from Theorem \ref{t2.3}(iv), \eqref{2.1}, and the Minkowski inequality
for the space $L^{p(\cdot)}(\rn)$, we deduce that
Corollary \ref{c2.1}(iv) holds true. However, since some weighted Lorentz
spaces and Lorentz--Morrey spaces may not be Banach spaces, it follows that (ii)
and (iii) of Corollary \ref{c2.1} can not be directly obtained
via using (ii) and (iii) of Theorem \ref{t2.3}, and the estimate \eqref{2.1}.

Indeed, (ii) of Corollary \ref{c2.1} can be proved via using Corollary \ref{c2.1}(i).
More precisely, by $\omega\in A_\fz(\rn)$ and the assumption \eqref{2.2}, we know that
there exists an $s_0\in(1,\fz)$ such that
$$\sup_{B\subset\rn}\lf(\fint_B \omega^{s_0}\,dx\r)^{\frac1{s_0}}
\lf\|v^{-\frac{1}{q_0}}\r\|^{q_0}_{\Phi,\,B}<\fz,
$$
which, combined with \cite[Theorem 6.13]{cmp11}, implies that
there exist a weight $u\in A_p(\rn)$ for any $p\in(q_0,\fz)$, and positive constants $c_1$ and $c_2$
such that $c_1\omega\le u\le c_2v$. From the special case of Corollary \ref{c2.1}(i),
we deduce that, for any $f\in L^{q-\uc_0}_u(\rn)$ and $h\in L^{q+\uc_0}_u(\rn)$,
$\|g^\ast_\lz(f)\|_{L^{q-\uc_0}_u(\rn)}\ls\|f\|_{L^{q-\uc_0}_u(\rn)}$ and
$\|g^\ast_\lz(h)\|_{L^{q+\uc_0}_u(\rn)}\ls\|h\|_{L^{q+\uc_0}_u(\rn)}$, where $\uc_0\in(0,q-q_0)$
is a constant, which, together with $c_1\omega\le u\le c_2v$, further implies that,
for any $f\in L^{q-\uc_0}_u(\rn)$ and $h\in L^{q+\uc_0}_u(\rn)$,
$\|g^\ast_\lz(f)\|_{L^{q-\uc_0}_\omega(\rn)}\ls\|f\|_{L^{q-\uc_0}_v(\rn)}$ and
$\|g^\ast_\lz(h)\|_{L^{q+\uc_0}_\omega(\rn)}\ls\|h\|_{L^{q+\uc_0}_v(\rn)}$.
By this and the Marcinkiewicz interpolation theorem of sublinear operators
in the scale of Lorentz spaces (see, for instance, \cite[Theorem 1.4.19]{g14}),
we conclude that Corollary \ref{c2.1}(ii) holds true. Choosing
the special weights $\omega$ and $v$, and using
Corollary \ref{c2.1}(ii), we obtain Corollary \ref{c2.1}(iii).

\begin{remark}\label{r2.3}
\begin{itemize}
\item[\rm(i)] The conclusion of Corollary \ref{c2.1}(i) was obtained by Cruz-Uribe and
P\'erez \cite[Theorem 1.7]{cp02} via a different method from that used in this article.
Indeed, Corollary \ref{c2.1}(i) was proved in \cite[Theorem 1.7]{cp02} by using the sharp maximal function
control for the Littlewood--Paley $g^\ast_\lz$-function.
We also point out that the range $\lz\in(2,\fz)$ for the Littlewood--Paley $g^\ast_\lz$-function
in Corollary \ref{c2.1}(i) coincides with that of \cite[Theorem 1.7]{cp02}.
Furthermore, it is worth pointing out that the one-weight boundedness of the Littlewood--Paley $g^\ast_\lz$-function
on $L^p_\omega(\rn)$, with $p\in(1,\fz)$ and $\omega\in A_p(\rn)$,
was established by Muckenhoupt and Wheeden \cite{mw74} (see also \cite{l14}).

Moreover, a two-weight boundedness of the Littlewood--Paley $g^\ast_\lz$-function
in the scale of weighted Lebesgue spaces was obtained in \cite{clx18} under the different
assumption on the weights $\omega$ and $v$. To the best of our knowledge, the two-weight boundedness of
the Lusin area function $S_\az$ in Theorem \ref{t2.3}(i) is new.
\item[\rm(ii)] To the best of our knowledge, the (two-weight) boundedness of
the Lusin area function $S_\az$ and the Littlewood--Paley $g^\ast_\lz$-function in the scale of weighted Lorentz spaces
and Lorentz--Morrey spaces presented in both (ii) and (iii) of Theorem \ref{t2.3} and (ii)
and (iii) of Corollary \ref{c2.1} is new.
\item[\rm(iii)] The boundedness of the Lusin area function $S_\az$ and
the Littlewood--Paley $g^\ast_\lz$-function on variable Lebesgue spaces is a simple corollary of
the one-weight boundedness of the Lusin area function $S_\az$ and
the Littlewood--Paley $g^\ast_\lz$-function in the scale of weighted Lebesgue spaces (see, for instance, \cite{l14,l11,w08}),
and the Rubio de Francia extrapolation theorem (see, for instance, \cite{cmp11}).
Thus, the conclusions of Theorem \ref{t2.3}(iv) and Corollary \ref{c2.1}(iv) are known.
\end{itemize}
\end{remark}

\begin{definition}\label{d2.3}
Let $\az\in(0,n)$ and $f\in\cs(\rn)$. The \emph{fractional integral}
$I_\az(f)$ is defined by setting, for any $x\in\rn$,
$$I_\az(f)(x)=c_{(\az,\,n)}\int_\rn\frac{f(y)}{|x-y|^{n-\az}}\,dy,
$$
where $c_{(\az,\,n)}$ is a positive constant depending only on $\az$ and $n$.
\end{definition}

Applying Theorem \ref{t1.3}, we obtain the following (two-weight) boundedness for fractional integral
operators.

\begin{theorem}\label{t2.4}
Let $\az\in(0,n)$, $q\in(1,\fz)$, $p\in(1,\frac{n}{\az})$ satisfy $\frac{1}{q}=\frac{1}{p}-\frac{\az}{n}$,
and $\Phi$ be a doubling Young function satisfying
$$\int_c^\fz\lf[\frac{t^{p'}}{\Phi(t)}\r]^{p-1}\frac{dt}{t}<\fz
$$
for some constant $c\in(0,\fz)$. Assume that the weights $\omega$ and $v$ satisfy
that $\omega\in A_\fz(\rn)$ and
$$\sup_{B\subset\rn}\lf(\fint_B \omega\,dx\r)\lf\|v^{-\frac{1}{p}}\r\|^{q}_{\Phi,\,B}<\fz,
$$
where the supremum is taken over all balls $B$ of $\rn$.
\begin{itemize}
\item[\rm(i)] Then $I_\az$ is bounded from $L^{p}_v(\rn)$ to $L^q_\omega(\rn)$,
and there exists a positive constant $C$ such that, for any $f\in L^{p}_v(\rn)$,
\begin{equation*}
\lf\|I_\az(f)\r\|_{L^q_\omega(\rn)}\le C\|f\|_{L^{p}_v(\rn)}.
\end{equation*}

\item[\rm(ii)] Let $\tz\in(0,n]$ and $\wz{\tz}:=n-\frac{p}{q}(n-\tz)$.
Then $I_\az$ is bounded from $\cm^{\wz{\tz}}_p(\rn)$ to $\cm^\tz_q(\rn)$,
and there exists a positive constant $C$ such that, for any $f\in \cm^{\wz{\tz}}_p(\rn)$,
\begin{equation*}
\lf\|I_\az(f)\r\|_{\cm^\tz_q(\rn)}\le C\|f\|_{\cm^{\wz{\tz}}_p(\rn)}.
\end{equation*}
\end{itemize}
\end{theorem}

We prove Theorem \ref{t2.4} in Section \ref{s4} below via using the boundedness of $I_\az$ from $L^1(\rn)$ to
$L^{\frac{n}{n-\az},\,\fz}(\rn)$, the Kolmogorov inequality, and Theorem \ref{t1.3}.

\begin{remark}\label{r2.4}
\begin{itemize}
\item[\rm(i)] The conclusion of Theorem \ref{t2.4}(i) is well known (see, for instance, \cite{p94,sw92}).
However, the proof of Theorem \ref{t2.4}(i) given in this article is different
from that used in \cite{p94,sw92}.
\item[\rm(ii)] Theorem \ref{t2.4}(ii) was established by Adams \cite{a73} (see also \cite[Theorem 7.1]{a15} and
\cite[Corollary 1.2]{sst11}).
We point out that the proof of Theorem \ref{t2.4}(ii) presented in this article is different from that used in
\cite{a15,a73,sst11}.
\end{itemize}
\end{remark}

\subsection{Applications to second-order elliptic operators of divergence form}\label{s2.2}

In this subsection, we establish the (two-weight) boundedness of Riesz transforms, Littlewood--Paley $g$-functions,
and fractional integral operators, associated with second-order divergence elliptic operators with complex
bounded measurable coefficients on $\rn$, in the scale of weighted Lebesgue spaces, weighted Lorentz spaces,
(Lorentz--)Morrey spaces, and variable Lebesgue spaces. We begin with recalling the
second-order divergence elliptic operators with complex
bounded measurable coefficients on $\rn$.

For any given $x\in\rn$, let $A(x):=\{a_{ij}(x)\}_{i,j=1}^n$ denote
an $n\times n$ matrix with complex-valued, bounded, measurable entries.
Then $A$ is said to satisfy the \emph{uniform ellipticity condition}
if there exists a positive constant $\mu_0\in(0,1]$ such that,
for any $x\in\rn$ and $\xi,\,\,\zeta\in\mathbb{C}^n$,
\begin{equation*}
\mu_0|\xi|^2\le\Re(A(x)\xi\cdot\overline{\xi})\quad \ \text{and}\quad \ |A(x)\xi\cdot\overline{\zeta}|\le \mu_0^{-1}
|\xi||\zeta|,
\end{equation*}
where $\Re(A(x)\xi\cdot\overline{\xi})$ denotes the \emph{real part} of $ A(x)\xi\cdot\xi$
and $\overline{\xi}$ denotes the \emph{conjugate vector} of $\xi$.
Denote by $L$ the \emph{maximal-accretive operator} on $L^2 (\rn)$ with largest domain
$D(L)\subset W^{1,2}(\rn)$ such that, for any $f\in D(L)$ and $g\in W^{1,2}(\rn)$,
\begin{equation*}
\langle Lf,g\rangle=\int_{\rn}A(x)\nabla f(x)\cdot\overline{\nabla g(x)}\,dx,
\end{equation*}
where $W^{1,2}(\rn)$ denotes the classical \emph{Sobolev space} on $\rn$ and $\nabla f$ the \emph{distributional
gradient} of $f$. In this sense, for any $f\in D(L)$, we write
\begin{equation}\label{2.3}
Lf:=-\divz(A\nabla f).
\end{equation}

It is well known that $-L$ generates a $C_0$-semigroup $\{e^{-tL}\}_{t>0}$ and $L$
has a bounded holomorphic functional calculus in $L^2(\rn)$ (see, for instance, \cite{a07}).
Moreover, the Riesz transform $\nabla L^{-1/2}$,
associated with the operator $L$, is defined by setting, for any $f\in L^2(\rn)$,
\begin{equation*}
\nabla L^{-1/2}(f):=\frac{1}{\pi^{1/2}}\int_0^\fz\nabla e^{-tL}(f)\frac{dt}{t^{1/2}}.
\end{equation*}

Let $(p_-(L),\,p_+(L))$ be the \emph{interior of the maximal interval of exponents}
 $p\in[1,\,\fz]$ for which the semigroup $\{e^{-tL}\}_{t>0}$ is bounded on $L^p(\rn)$.
Furthermore, denote by $(q_-(L),\,q_+(L))$ the \emph{interior of
the maximal interval of exponents} $p\in[1,\,\fz]$ for which the family of operators, $\{\sqrt t \nabla
e^{-tL}\}_{t>0}$, is bounded on $L^p(\rn)$.

\begin{remark}\label{r2.5}
Let the operator $L$ be as in \eqref{2.3}.
It is well known that, when $n\in\{1,\,2\}$, $p_-(L)=1$ and $p_+(L)=\fz$;
when $n\in[3,\fz)\cap\nn$, $p_-(L)\in[1,\frac{2n}{n+2})$ and $p_+(L)\in(\frac{2n}{n-2},\fz]$
(see, for instance, \cite[Corollary 3.6]{a07}).
Moreover, for $q_-(L)$ and $q_+(L)$,
it is also known that, when $n=1$, $q_-(L)=1$ and $q_+(L)=\fz$;
when $n=2$, $q_-(L)=1$ and $q_+(L)\in(2,\fz]$; when $n\in[3,\fz)\cap\nn$,
$q_-(L)\in[1,\frac{2n}{n+2})$ and $q_+(L)\in(2,\fz]$
(see, for instance, \cite[Section 3.4]{a07}).
\end{remark}

Then we have the following (two-weight) boundedness for the Riesz transform $\nabla L^{-1/2}$
in the scale of weighted Lebesgue spaces, weighted Lorentz spaces, (Lorentz--)Morrey spaces,
and variable Lebesgue spaces.

\begin{theorem}\label{t2.5}
Let $L$ be as in \eqref{2.3} and $q\in(q_-(L),\,q_+(L))$.
Assume that the weights $\omega$ and $v$ satisfy
that $\omega\in RH_s(\rn)$ with some $s\in((\frac{q_+(L)}{q})',\fz]$,
$v^{1-(\frac{q}{q_-(L)})'}\in A_\fz(\rn)$, and
\begin{equation}\label{2.4}
\lf[\omega,v^{1-(\frac{q}{q_-(L)})'}\r]_{A_{\frac{q}{q_-(L)}}(\rn)}:=\sup_{B\subset\rn}
\lf[\fint_B \omega\,dx\r]\lf[\fint_B v^{1-(\frac{q}{q_-(L)})'}\,dx\r]^{\frac{q}{q_-(L)}-1}<\fz,
\end{equation}
where the supremum is taken over all balls $B$ of $\rn$.
\begin{itemize}
\item[\rm(i)] Then $\nabla L^{-1/2}$ is bounded from $L^q_v(\rn)$ to $L^q_\omega(\rn)$,
and there exists a positive constant $C$ such that, for any $f\in L^q_v(\rn)$,
\begin{equation*}
\lf\|\nabla L^{-1/2}(f)\r\|_{L^q_\omega(\rn)}\le C\|f\|_{L^q_v(\rn)}.
\end{equation*}

\item[\rm(ii)] Let $t\in(0,\fz]$. Then $\nabla L^{-1/2}$ is bounded from $L^{q,\,t}_v(\rn)$ to $L^{q,\,t}_\omega(\rn)$,
and there exists a positive constant $C$ such that, for any $f\in L^{q,\,t}_v(\rn)$,
\begin{equation*}
\lf\|\nabla L^{-1/2}(f)\r\|_{L^{q,\,t}_\omega(\rn)}\le C\|f\|_{L^{q,\,t}_v(\rn)}.
\end{equation*}

\item[\rm(iii)] Let $t\in(0,\fz]$ and $\tz\in(\frac{nq}{q_+(L)},n]$.
Then $\nabla L^{-1/2}$ is bounded on $L^{q,\,t;\,\tz}(\rn)$, and there exists a positive constant $C$ such that,
for any $f\in L^{q,\,t;\,\tz}(\rn)$,
\begin{equation*}
\lf\|\nabla L^{-1/2}(f)\r\|_{L^{q,\,t;\,\tz}(\rn)}\le C\|f\|_{L^{q,\,t;,\tz}(\rn)}.
\end{equation*}
In particular, $\nabla L^{-1/2}$ is bounded on $\cm^{\tz}_q(\rn)$,
and there exists a positive constant $C$ such that, for any $f\in \cm^{\tz}_q(\rn)$,
\begin{equation*}
\lf\|\nabla L^{-1/2}(f)\r\|_{\cm^{\tz}_q(\rn)}\le C\|f\|_{\cm^{\tz}_q(\rn)}.
\end{equation*}

\item[\rm(iv)] Assume that $p(\cdot)\in\mathcal{P}(\rn)$ satisfies $q_-(L)<p_-\le p_+\le q_+(L)$
with $p_-$ and $p_+$ as in \eqref{1.7}, and that $\cm$ is bounded on $L^{p(\cdot)}(\rn)$.
Then $\nabla L^{-1/2}$ is bounded on $L^{p(\cdot)}(\rn)$, and there exists a positive constant $C$ such that,
for any $f\in L^{p(\cdot)}(\rn)$,
\begin{equation*}
\lf\|\nabla L^{-1/2}(f)\r\|_{L^{p(\cdot)}(\rn)}\le C\|f\|_{L^{p(\cdot)}(\rn)}.
\end{equation*}
\end{itemize}
\end{theorem}

\begin{remark}\label{r2.6}
\begin{itemize}
\item[\rm(i)] The conclusion of Theorem \ref{t2.5}(i) with $\omega=v$ was obtained
by Auscher and Martell \cite[Theorem 5.2]{am06}. Thus, Theorem \ref{t2.5}(i) generalizes the one-weight
boundedness of the Riesz transform $\nabla L^{-1/2}$ established by Auscher and Martell \cite[Theorem 5.2]{am06}
to the two-weight case.
\item[\rm(ii)] Using a different method from that used in this article,
Sawano et al. \cite[Theorem 1.8]{shs19} established the boundedness of $\nabla L^{-1/2}$
on Morrey spaces $\cm^{\tz}_q(\rn)$ under the different assumptions
from those of Theorem \ref{t2.5}(iii) for $\tz$ and $q$.
To the best of our knowledge, the (two-weight) boundedness of
the Riesz transform $\nabla L^{-1/2}$ in the scale of weighted Lorentz spaces,
Lorentz--Morrey spaces, and variable Lebesgue spaces obtained in Theorem \ref{t2.5} is new.
\end{itemize}
\end{remark}

Let $L$ be as in \eqref{2.3}. For any $f\in L^2(\rn)$, the \emph{Littlewood--Paley $g$-function} $G_L(f)$,
associated with $L$, is defined by setting, for any $x\in\rn$,
\begin{equation*}
G_L(f)(x):=\lf[\int_0^\fz\lf|\nabla e^{-tL}f(x)\r|^2\,dt\r]^{\frac12}.
\end{equation*}

For the Littlewood--Paley $g$-function associated with $L$, we have the following (two-weight) boundedness.

\begin{theorem}\label{t2.6}
Let $L$ be as in \eqref{2.3} and $q\in(q_-(L),\,q_+(L))$. Assume that the weights $\omega$ and $v$ satisfy
$\omega\in RH_s(\rn)$ with some $s\in((\frac{q_+(L)}{q})',\fz]$,
$v^{1-(\frac{q}{q_-(L)})'}\in A_\fz(\rn)$, and \eqref{2.4}.
\begin{itemize}
\item[\rm(i)] Then there exists a positive constant $C$ such that, for any $f\in L^q_v(\rn)$,
\begin{equation*}
\lf\|G_L(f)\r\|_{L^q_\omega(\rn)}\le C\|f\|_{L^q_v(\rn)}.
\end{equation*}

\item[\rm(ii)] Let $t\in(0,\fz]$. Then there exists a positive constant $C$ such that,
 for any $f\in L^{q,\,t}_v(\rn)$,
\begin{equation*}
\lf\|G_L(f)\r\|_{L^{q,\,t}_\omega(\rn)}\le C\|f\|_{L^{q,\,t}_v(\rn)}.
\end{equation*}

\item[\rm(iii)] Let $t\in(0,\fz]$ and $\tz\in(\frac{nq}{q_+(L)},n]$.
Then there exists a positive constant $C$ such that, for any $f\in L^{q,\,t;\,\tz}(\rn)$,
\begin{equation*}
\lf\|G_L(f)\r\|_{L^{q,\,t;\,\tz}(\rn)}\le C\|f\|_{L^{q,\,t;\,\tz}(\rn)}.
\end{equation*}
In particular, there exists a positive constant $C$ such that, for any $f\in \cm^{\tz}_q(\rn)$,
\begin{equation*}
\lf\|G_L(f)\r\|_{\cm^{\tz}_q(\rn)}\le C\|f\|_{\cm^{\tz}_q(\rn)}.
\end{equation*}

\item[\rm(iv)] Assume that $p(\cdot)\in\mathcal{P}(\rn)$ satisfies $q_-(L)<p_-\le p_+\le q_+(L)$
with $p_-$ and $p_+$ as in \eqref{1.7}, and that $\cm$ is bounded on $L^{p(\cdot)}(\rn)$.
Then there exists a positive constant $C$ such that, for any $f\in L^{p(\cdot)}(\rn)$,
\begin{equation*}
\lf\|G_L(f)\r\|_{L^{p(\cdot)}(\rn)}\le C\|f\|_{L^{p(\cdot)}(\rn)}.
\end{equation*}
\end{itemize}
\end{theorem}

\begin{remark}\label{r2.7}
\begin{itemize}
\item[\rm(i)] Theorem \ref{t2.6}(i) with $\omega=v$ was established
by Auscher and Martell \cite[Theorem 7.2(b)]{am06}. Therefore, Theorem \ref{t2.6}(i) extends the one-weight
boundedness of the Littlewood--Paley $g$-function associated with the operator $L$ obtained in \cite[Theorem 7.2(b)]{am06}
to the two-weight case.
\item[\rm(ii)] To the best of our knowledge, the (two-weight) boundedness of
the Littlewood--Paley $g$-function associated with the operator $L$ in the scale of weighted Lorentz spaces,
(Lorentz--)Morrey spaces, and variable Lebesgue spaces obtained in Theorem \ref{t2.6} is new.
\end{itemize}
\end{remark}

Let $\az\in(0,n)$. For any $f\in \cs(\rn)$, the fractional integral $L^{-\az/2}(f)$,
associated with $L$, is defined by setting, for any $x\in\rn$,
$$
L^{-\az/2}(f)(x):=\frac{1}{\Gamma(\az/2)}\int_0^\fz t^{\az/2}e^{-tL}f(x)\frac{dt}{t},
$$
where $\Gamma(\cdot)$ denotes the usual Gamma function.

Applying Theorem \ref{t1.3}, we obtain the following (two-weight) boundedness for the fractional integral
operator $L^{-\az/2}$.

\begin{theorem}\label{t2.7}
Let $L$ be as in \eqref{2.3}, $\az\in(0,\frac{n}{p_-(L)})$, $q\in(\wz{p}_-(L),p_+(L))$, and $p\in(p_-(L),p_+(L))$ satisfy $\frac{1}{q}=\frac{1}{p}-\frac{\az}{n}$,
where $\wz{p}_-(L)$ is given by setting $\frac{1}{p_-(L)}-\frac{1}{\wz{p}_-(L)}=\frac{\az}{n}$.
Assume that the weights $\omega$ and $v$ satisfy
that $\omega\in RH_s(\rn)$ with some $s\in((\frac{p_+(L)}{q})',\fz]$, $v^{1-(\frac{p}{p_-(L)})'}\in A_\fz(\rn)$, and
\begin{equation}\label{2.5}
\lf[\omega,v^{1-(\frac{p}{p_-(L)})'}\r]_{A_{\frac{p_-(L)}{q}(1-\frac{\az p_-(L)}{n})}(\rn)}:=
\sup_{B\subset\rn}\lf[\fint_B \omega\,dx\r]
\lf[\fint_B v^{1-(\frac{p}{p_-(L)})'}\,dx\r]^{\frac{p_-(L)}{q}(1-\frac{\az p_-(L)}{n})-1}<\fz,
\end{equation}
where the supremum is taken over all balls $B$ of $\rn$.
\begin{itemize}
\item[\rm(i)] Then $L^{-\az/2}$ is bounded from $L^p_v(\rn)$ to $L^q_\omega(\rn)$, and there exists
a positive constant $C$ such that, for any $f\in L^p_v(\rn)$,
\begin{equation*}
\lf\|L^{-\az/2}(f)\r\|_{L^q_\omega(\rn)}\le C\|f\|_{L^{p}_v(\rn)}.
\end{equation*}

\item[\rm(ii)] Let $\tz\in(\frac{nq}{p_+(L)},n]$ and $\wz{\tz}:=n-\frac{p}{q}(n-\tz)$.
Then $L^{-\az/2}$ is bounded from $\cm^{\wz{\tz}}_p(\rn)$ to $\cm^\tz_q(\rn)$, and
there exists a positive constant $C$ such that, for any $f\in \cm^{\wz{\tz}}_p(\rn)$,
\begin{equation*}
\lf\|L^{-\az/2}(f)\r\|_{\cm^\tz_q(\rn)}\le C\|f\|_{\cm^{\wz{\tz}}_p(\rn)}.
\end{equation*}
\end{itemize}
\end{theorem}

\begin{remark}\label{r2.8}
\begin{itemize}
\item[\rm(i)] The conclusion of Theorem \ref{t2.7}(i) with $\omega=v$ was obtained
by Auscher and Martell \cite[Theorem 1.3]{am08}. Thus, Theorem \ref{t2.7}(i) generalizes the one-weight
boundedness of the fractional integral operator $L^{-\az/2}$ established in \cite[Theorem 1.3]{am08}
to the two-weight case.
\item[\rm(ii)] To the best of our knowledge, the boundedness of the fractional integral operator $L^{-\az/2}$
in the scale of Morrey spaces obtained in Theorem \ref{t2.7}(ii) is new.
\end{itemize}
\end{remark}

\section{Proofs of Theorems \ref{t1.1}, \ref{t1.2}, and \ref{t1.3}\label{s3}}

 In this section, we first give the proof of Theorem \ref{t1.1}; we then
show Theorems \ref{t1.2} and \ref{t1.3} via using Theorem \ref{t1.1} and
the two-weight inequality for the (fractional) Hardy--Littlewood maximal
function.

To show Theorem \ref{t1.1}, we need the following properties of $A_p(\rn)$ weights,
which are well known (see, for instance, \cite[Chapter 7]{g14}).

\begin{lemma}\label{l3.1}
\begin{enumerate}
\item[\rm(i)] $\bigcup_{p\in[1,\fz)}A_p(\rn)=A_\fz(\rn)
=\bigcup_{r\in(1,\fz]}RH_r(\rn).$
\item[\rm(ii)] If $\omega\in RH_s(\rn)$ with $s\in(1,\fz]$, then,
for any ball $B$ of $\rn$ and any measurable set $E\subset B$,
$$\frac{\omega(E)}{\omega(B)}\le [\omega]_{RH_s(\rn)}
\lf[\frac{|E|}{|B|}\r]^{\frac{s-1}s}.$$
\item[\rm(iii)] Let $\omega\in A_p(\rn)$ with some $p\in[1,\fz)$ and $\gz\in(0,1)$.
Then $\omega^{\gz}\in RH_{\gz^{-1}}(\rn)$, and there exists a positive constant
$C$, depending only on $[\omega]_{A_p(\rn)}$ and $\gz$, such that
$[\omega^\gz]_{RH_{\gz^{-1}}(\rn)}\le C$.
\end{enumerate}
\end{lemma}

Moreover, we also need the following Whitney covering lemma (see, for instance, \cite[p.\,15, Lemma 2]{St93}).

\begin{lemma}\label{l3.2}
Let $\boz$ be an open proper subset of $\rn$.
Then there exist two sequences $\{x_i:\ x_i\in\boz\}_i$ and $\{r_i:\ r_i\in(0,\fz)\}_i$ such that
\begin{enumerate}
\item[\rm(i)] $\boz=\bigcup_i B(x_i,r_i)$, but $\{B(x_i,r_i/4)\}_i$ are
mutually disjoint balls;
\item[\rm(ii)] for any $i$, $B(x_i,2r_i)\cap\boz^\complement=\emptyset$,
but $B(x_i,4r_i)\cap\boz^\complement\neq\emptyset$;
\item[\rm(iii)] there exists a positive constant $N:=N(n)$, depending only on
$n$, such that, for any $x\in\boz$,
$$
\sum_i\mathbf{1}_{B(x_i,2r_i)}(x)\le N.
$$
\end{enumerate}
\end{lemma}

\begin{lemma}\label{l3.3}
There exists a positive constant $K_0:=K_0(n)$, depending only on $n$,
such that, for any $f\in L^1_\loc(\rn)$, if any given ball $B$ of $\rn$
and $\lz\in(0,\fz)$ satisfy that there exists an $x_0\in 4B$ such that
$\cm(f)(x_0)\le\lz$, then, for any $K\in[K_0,\fz)$,
$$
\lf\{x\in\rn:\ \mathbf{1}_B(x)\cm(f)(x)>K\lz\r\}
\subset\lf\{x\in\rn:\ \cm(f\mathbf{1}_{8B})(x)
>\frac{K}{K_0}\lz\r\}.
$$
\end{lemma}

Lemma \ref{l3.3} is just \cite[Lemma 1.4]{a07}.

Furthermore, to prove Theorem \ref{t1.1}, we need the following (weak) boundedness of the Hardy--Littlewood
maximal operator $\cm$ on $L^p(\rn)$ (see, for instance, \cite{g14,St93}).

\begin{lemma}\label{l3.4}
Let $\cm$ be the Hardy--Littlewood maximal operator on $\rn$ and $p\in[1,\fz)$.
Then there exists a positive constant $C_{(n,\,p)}$, depending only on $n$ and $p$,
such that, for any $f\in L^p(\rn)$,
$$\sup_{\lz\in(0,\fz)}\lf\{\lz\lf|\lf\{x\in\rn:\
\cm(f)(x)>\lz\r\}\r|^{\frac1p}\r\}
\le C_{(n,\,p)}\|f\|_{L^p(\rn)}.
$$
Moreover, if $p\in(1,\fz)$, then, for any $f\in L^p(\rn)$,
$\|\cm(f)\|_{L^p(\rn)}\le C_{(n,\,p)}\|f\|_{L^p(\rn)}$.
\end{lemma}

Now, we show Theorem \ref{t1.1} by using Lemmas \ref{l3.1}
through \ref{l3.4}.

\begin{proof}[Proof of Theorem \ref{t1.1}]
For any $\lz\in(0,\fz)$, let
$$E(\lz):=\lf\{x\in\rn:\ \cm\lf(|F|^{p_1}\r)(x)>\lz\r\}.
$$
Without loss of generality, we may assume that $E(\lz)$ is a proper subset of $\rn$.
Since $E(\lz)$ is open, it follows from the Whitney covering lemma
(see Lemma \ref{l3.2}) that there exist a family $\{B_i\}_i$ of balls
and a positive constant $N$, depending only on $n$, such that
$\bigcup_iB_i=E(\lz)$, $\{\frac14B_i\}_i$ are mutually disjoint balls,
$\sum_i\mathbf{1}_{2B_i}\le N$, and $(2B_i)\cap
[E(\lz)]^\complement=\emptyset$ and $(4B_i)\cap [E(\lz)]^\complement\neq\emptyset$
for any $i$. Thus, for any $i$, there exists a $y_i\in 4B_i$ such that $\cm(|F|^{p_1})(y_i)\le\lz$.

Let
$$H_\lz:=\lf\{x\in\rn:\ \cm\lf(|F|^{p_1}\r)(x)>\bz\lz,\
\cm_\gz\lf(|f|^{p_2}\r)(x)\le(\kappa\lz)^{\frac{p_2}{p_1}}\r\},$$
where $\bz\in(1,\fz)$ and $\kappa\in(0,1)$ are constants determined later.
Then $H_\lz\subset E(\lz)$ and hence
\begin{equation*}
\omega(H_\lz)\le\sum_{i}\omega(H_\lz\cap B_i).
\end{equation*}

Next, we fix $i$. Assume that $H_\lz\cap B_i\neq\emptyset$; otherwise, we have
nothing to do. Then there exists an $x_i\in B_i$ such that $\cm_\gz(|f|^{p_2})(x_i)
\le(\kappa\lz)^{p_2/p_1}$. By the assumption of Theorem \ref{t1.1}, we conclude that
there exist two measurable functions $F_{8B_i}$ and $R_{8B_i}$ on $8B_i$ such that
$|F|\le |F_{8B_i}|+|R_{8B_i}|$ on $8B_i$. From this and Lemma \ref{l3.3}, we deduce that
there exists a positive constant $\bz_1$ such that, if $\bz\ge\bz_1$, then
\begin{align}\label{3.1}
|H_\lz\cap B_i|&\le|\lf\{x\in\rn:\ \cm\lf(|F|^{p_1}\r)(x)>\bz\lz\r\}\cap B_i| \\
&\le\lf|\lf\{x\in\rn:\ \cm\lf(|F|^{p_1}\mathbf{1}_{8B_i}\r)(x)>\frac{\bz}{\bz_1}\lz\r\}\r|\nonumber \\
&\le\lf|\lf\{x\in\rn:\ \cm\lf(|F_{8B_i}|^{p_1}\mathbf{1}_{8B_i}\r)(x)>\frac{\bz}{2^{p_1+1}\bz_1}\lz\r\}\r|\nonumber \\
&\quad+\lf|\lf\{x\in\rn:\ \cm\lf(|R_{8B_i}|^{p_1}\mathbf{1}_{8B_i}\r)(x)>\frac{\bz}{2^{p_1+1}\bz_1}\lz\r\}\r|.\nonumber
\end{align}
We first assume that $p_3\in(0,\fz)$. By the weak boundedness of the maximal operator $\cm$
(see Lemma \ref{l3.4}) and \eqref{1.5}, we find that
\begin{align}\label{3.2}
&\lf|\lf\{x\in\rn:\ \cm\lf(|F_{8B_i}|^{p_1}\mathbf{1}_{8B_i}\r)(x)>\frac{\bz}{2^{p_1+1}\bz_1}\lz\r\}\r|\\
&\quad\le\frac{2^{p_1+1}C_{(n,\,1)}\bz_1}{\bz\lz}\int_{8B_i}|F_{8B_i}|^{p_1}\,dx \nonumber \\
&\quad\le\frac{2^{p_1+1}C_{(n,\,1)}\bz_1}{\bz\lz}|8B_i|\lf[\epsilon\lz^{\frac{1}{p_1}}
+C_2(\kappa\lz)^{\frac{1}{p_1}}\r]^{p_1}\nonumber \\
&\quad=\frac{2^{p_1}C_{(n)}\bz_1}{\bz}|B_i|\lf(\epsilon+C_2\kappa^{\frac{1}{p_1}}\r)^{p_1}.\nonumber
\end{align}
Moreover, from \eqref{1.4} and Lemma \ref{l3.4}, it follows that
\begin{align}\label{3.3}
&\lf|\lf\{x\in\rn:\ \cm\lf(|R_{8B_i}|^{p_1}\mathbf{1}_{8B_i}\r)(x)>\frac{\bz}{2^{p_1+1}\bz_1}\lz\r\}\r| \\
&\quad\le\lf[\frac{2^{p_1+1}C_{(n,\,p_3/p_1)}\bz_1}{\bz\lz}\r]^{\frac{p_3}{p_1}}\int_{8B_i}|R_{8B_i}|^{p_3}\,dx\nonumber \\
&\quad\le\lf[\frac{2^{p_1+1}C_{(n,\,p_3/p_1)}\bz_1}{\bz\lz}\r]^{\frac{p_3}{p_1}}|8B_i|
\lf[C_1\lz^{\frac{1}{p_1}}+C_1(\kappa\lz)^{\frac{1}{p_1}}\r]^{p_3}\nonumber \\
&\quad=\frac{2^{p_3}C_{1}^{p_3}C_{(n,\,p_1,\,p_3)}\bz_1^{p_3/p_1}}{\bz^{p_3/p_1}}|B_i|
\lf(1+\kappa^{\frac{1}{p_1}}\r)^{p_3}.\nonumber
\end{align}
Then, by \eqref{3.1}, \eqref{3.2}, and \eqref{3.3}, we conclude that
\begin{align}\label{3.4}
|H_\lz\cap B_i|&\le\lf[\frac{2^{p_1}C_{(n)}\bz_1}{\bz}\lf(\epsilon+C_2\kappa^{\frac{1}{p_1}}\r)^{p_1}+
\frac{2^{p_3}C_1^{p_3}C_{(n,\,p_1,\,p_3)}\bz_1^{p_3/p_1}}{\bz^{p_3/p_1}}
\lf(1+\kappa^{\frac{1}{p_1}}\r)^{p_3}\r]|B_i| \\
&=\lf[\frac{2^{p_1}C_{(n)}\bz_1}{\bz^{1-a}}\lf(\epsilon+C_2\kappa^{\frac{1}{p_1}}\r)^{p_1}+
\frac{2^{p_3}C_1^{p_3}C_{(n,\,p_1,\,p_3)}\bz_1^{p_3/p_1}}{\bz^{p_3/p_1-a}}
\lf(1+\kappa^{\frac{1}{p_1}}\r)^{p_3}\r]\bz^{-a}|B_i|.\nonumber
\end{align}
From the assumption $a\in(1,\frac{p_3}{p_1})$, we deduce that there exists a positive
constant $\bz_0\ge\max\{1,\,\bz_1\}$ such that, if $\bz\ge\bz_0$, then
\begin{align}\label{3.5}
\frac{2^{p_3}C_1^{p_3}C_{(n,\,p_1,\,p_3)}\bz_1^{p_3/p_1}}{\bz^{p_3/p_1-a}}
\lf(1+\kappa^{\frac{1}{p_1}}\r)^{p_3}\le\frac{1}{2}\lf(\frac{1}{N}\r)^{\frac{s}{s-1}}
[\omega]_{RH_s(\rn)}^{-\frac{s}{s-1}}.
\end{align}
Fix $\bz$. Take $\epsilon$ and $\kappa$ small enough such that
\begin{align}\label{3.6}
\frac{2^{p_1}C_{(n)}\bz_1}{\bz^{1-a}}\lf(\epsilon+C_2\kappa^{\frac{1}{p_1}}\r)^{p_1}
\le\frac{1}{2}\lf(\frac{1}{N}\r)^{\frac{s}{s-1}}
[\omega]_{RH_s(\rn)}^{-\frac{s}{s-1}}.
\end{align}
Then, by \eqref{3.4}, \eqref{3.5}, \eqref{3.6}, and Lemma \ref{l3.1}(ii),
we know that, for any $i$,
\begin{align*}
\frac{\omega(H_\lz\cap B_i)}{\omega(B_i)}&\le[\omega]_{RH_s(\rn)}
\lf(\frac{|H_\lz\cap B_i|}{|B_i|}\r)^{\frac{s-1}{s}}\\
&\le[\omega]_{RH_s(\rn)}[\omega]_{RH_s(\rn)}^{-1}N^{-1}\bz^{-\frac{a(s-1)}{s}}
=N^{-1}\bz^{-\frac{a(s-1)}{s}},
\end{align*}
which, combined with $\sum_i\mathbf{1}_{B_i}\le\sum_i\mathbf{1}_{2B_i}\le N$, further implies that
\begin{align*}
\omega(H_\lz)&\le\sum_i\omega(H_\lz\cap B_i)\le N^{-1}\bz^{-\frac{a(s-1)}s}
\sum_i\omega(B_i)=N^{-1}\bz^{-\frac{a(s-1)}s}
\sum_i\int_\rn\omega\mathbf{1}_{B_i}\,dx\\
&\le\bz^{-\frac{a(s-1)}s}\int_\rn\omega\mathbf{1}_{E(\lz)}\,dx
=\bz^{-\frac{a(s-1)}s}\omega(E(\lz)).
\end{align*}
From this and the definition of $H_\lz$, it follows that, if $\bz\ge\bz_0$,
then, for any $\lz\in(0,\fz)$,
\begin{align}\label{3.7}
\omega(E(\bz\lz))\le\bz^{-\frac{a(s-1)}s}\omega(E(\lz))+
\omega\lf(\lf\{x\in\rn:\ \cm_\gz(|f|^{p_2})(x)>(\kappa\lz)^{\frac{p_2}{p_1}}\r\}\r).
\end{align}
This finishes the proof of Theorem \ref{t1.1} in the case of $p_3\in(0,\fz)$.

Finally, we consider the case that $p_3=\fz$. In this case, by \eqref{1.4}, we conclude that
\begin{align*}
&\lf\|\cm\lf(|R_{8B_i}|^{p_1}\mathbf{1}_{8B_i}\r)\r\|_{L^\fz(\rn)}\\
&\quad\le\lf\||R_{8B_i}|^{p_1}\mathbf{1}_{8B_i}\r\|_{L^\fz(\rn)}\\ \nonumber
&\quad\le C_1^{p_1} \lf\{\lf[\cm(|F|^{p_1})(y_i)\r]^{\frac{1}{p_1}}
+\lf[\cm_\gz(|f|^{p_2})(x_i)\r]^{\frac{1}{p_2}}\r\}^{p_1}
\le C_1^{p_1}\lf(1+\kappa^{\frac{1}{p_1}}\r)\lz.
\end{align*}
Take $\bz_0\in(1,\fz)$ large enough such that, if $\bz\ge\bz_0$, then
$$\frac{\bz}{2^{p_1+1}\bz_1}>C_1^{p_1}\lf(1+\kappa^{\frac{1}{p_1}}\r).
$$
Thus, when $\bz\ge\bz_0$,
$$
\lf|\lf\{x\in\rn:\ \cm\lf(|R_{8B_i}|^{p_1}\mathbf{1}_{8B_i}\r)(x)>\frac{\bz}{2^{p_1+1}\bz_1}\lz\r\}\r|=0.
$$
Replacing \eqref{3.3} by this estimate, and repeating the proof of \eqref{3.7}, we then complete
the proof in the case of $p_3=\fz$ and hence of Theorem \ref{t1.1}.
\end{proof}

To prove Theorem \ref{t1.2} by using Theorem \ref{t1.1}, we need the following
two-weight boundedness of the Hardy--Littlewood maximal operator $\cm$ in the scale
of weighed Lebesgue spaces and weighted Lorentz spaces.

\begin{lemma}\label{l3.5}
Let  $p\in(1,\fz)$ and $\Phi$ be a doubling Young function
satisfying
$$\int_c^\fz\lf[\frac{t^{p'}}{\Phi(t)}\r]^{p-1}\frac{dt}{t}<\fz
$$
for some constant $c\in(0,\fz)$. Assume that the weights $\omega$ and $v$ satisfy
that
$$\sup_{B\subset\rn}\lf(\fint_B \omega\,dx\r)\lf\|v^{-\frac{1}{p}}\r\|^{p}_{\Phi,\,B}<\fz,
$$
where the supremum is taken over all balls $B$ of $\rn$.
\begin{itemize}
\item[\rm(i)] Then $\cm$ is bounded from $L^p_v(\rn)$ to $L^p_\omega(\rn)$.
\item[\rm(ii)] Let $r\in(0,\fz]$. Assume further that $\omega\in A_\fz(\rn)$, there exists a $p_0\in(1,p)$
such that $\int_c^\fz[\frac{t^{p_0'}}{\Phi(t)}]^{p_0-1}\frac{dt}{t}<\fz$ for some constant $c\in(0,\fz)$,
and
$$\sup_{B\subset\rn}\lf(\fint_B \omega\,dx\r)\lf\|v^{-\frac{1}{p_0}}\r\|^{p_0}_{\Phi,\,B}<\fz,
$$
where the supremum is taken over all balls $B$ of $\rn$.
Then $\cm$ is bounded from $L^{p,\,r}_v(\rn)$
to $L^{p,\,r}_\omega(\rn)$.
\end{itemize}
\end{lemma}

\begin{proof}
The conclusion of (i) is well known (see, for instance, \cite[Theorem 5.14]{cmp11} and \cite[Theorem 1.2]{p95}).
Here we only give the proof of (ii).

By $\omega\in A_\fz(\rn)$ and Lemma \ref{l3.1}(i), we conclude that there exists an $s\in(1,\fz)$ such that
$$\sup_{B\subset\rn}\lf(\fint_B \omega^s\,dx\r)^{\frac1s}\lf\|v^{-\frac{1}{p_0}}\r\|^{p_0}_{\Phi,\,B}<\fz,
$$
where the supremum is taken over all balls $B$ of $\rn$. From this and \cite[Theorem 6.13]{cmp11}, we deduce that
there exist a weight $u\in A_q(\rn)$ for any $q\in(p_0,\fz)$ and positive constants $c_1$ and $c_2$
such that $c_1\omega\le u\le c_2v$. By the one-weight boundedness of $\cm$ (see, for instance, \cite{cmp11,g14})
and the assumption that $u\in A_q(\rn)$ for any $q\in(p_0,\fz)$, we know that
$\cm$ is bounded on both $L^{p-\uc_0}_u(\rn)$ and $L^{p+\uc_0}_u(\rn)$, where $\uc_0\in(0,p-p_0)$
is a constant, which, combined with $c_1\omega\le u\le c_2v$, further implies that
$\cm$ is bounded from $L^{p-\uc_0}_v(\rn)$ to $L^{p-\uc_0}_\omega(\rn)$, and
from $L^{p+\uc_0}_v(\rn)$ to $L^{p+\uc_0}_\omega(\rn)$.
From this and the Marcinkiewicz interpolation theorem of sublinear operators
in the scale of Lorentz spaces (see, for instance, \cite[Theorem 1.4.19]{g14}),
it follows that $\cm$ is bounded from $L^{p,\,r}_v(\rn)$ to $L^{p,\,r}_\omega(\rn)$.
This finishes the proof of (ii) and hence of Lemma \ref{l3.5}.
\end{proof}

To show Theorem \ref{t1.2} via using Theorem \ref{t1.1},
we need the following lemma, which is well known (see, for instance,
\cite[Section 7.1.2]{g14} and \cite[Lemma 3.4]{mp12}).

\begin{lemma}\label{l3.6}
\begin{itemize}
\item[{\rm(i)}] Let $s\in[1,\fz)$, $\omega\in A_s(\rn)$, $z\in\rn$, and $k\in(0,\fz)$
be a constant. Assume that $\tau^z(\omega)(\cdot):=\omega(\cdot-z)$ and
$\omega_k:=\min\{\omega,\,k\}$. Then $\tau^z(\omega)\in A_s(\rn)$ and
$[\tau^z(\omega)]_{A_s(\rn)}=[\omega]_{A_s(\rn)}$, and
$\omega_k\in A_s(\rn)$ and $[\omega_k]_{A_s(\rn)}\le c_{(s)}[\omega]_{A_s(\rn)}$,
where $c_{(s)}:=1$ when $s\in[1,2]$, and $c_{(s)}:=2^{s-1}$ when $s\in(2,\fz)$.
\item[{\rm(ii)}] For any $x\in\rn$, let $\omega_\gz(x):=|x|^\gz$,
where $\gz\in\rr$ is a constant. Then, for any given $s\in(1,\fz)$, $\omega_\gz\in A_s(\rn)$
if and only if $\gz\in(-n,n[s-1])$. Moreover, $[\omega_\gz]_{A_s(\rn)}\le C_{(n,\,s,\,\gz)}$,
where $C_{(n,\,s,\,\gz)}$ is a positive constant depending only on $n$, $s$, and $\gz$.
\end{itemize}
\end{lemma}

Moreover, we also need the following limited range extrapolation theorem in the
variable exponent case, which was established by Cruz-Uribe and Wang in \cite[Theorem 2.14]{cw17}.

\begin{lemma}\label{l3.7}
Let $f$ and $h$ be two given non-negative measurable functions on $\rn$,
and $1<q_1<p<q_2<\fz$.
Assume that, for any given $\omega\in A_{p/q_1}(\rn)\cap RH_{(q_2/p)'}(\rn)$,
\begin{equation*}
\|f\|_{L^p_\omega(\rn)}\le C\|h\|_{L^p_\omega(\rn)},
\end{equation*}
where $C$ is a positive constant depending only
on $n$, $p$, $[\omega]_{A_{p/q_1}(\rn)}$, and $[\omega]_{RH_{(q_2/p)'}(\rn)}$.
Assume further that $p(\cdot)\in\cp(\rn)$ satisfies that $\cm$ is bounded on $L^{p(\cdot)}(\rn)$
and $q_1<p_-\le p_+<q_2$, where $p_-$ and $p_+$ are as in \eqref{1.7}.
Then there exists a positive constant $C$,
depending only on $n$ and $p(\cdot)$, such that
$\|f\|_{L^{p(\cdot)}(\rn)}\le C\|h\|_{L^{p(\cdot)}(\rn)}$.
\end{lemma}

Now, we show Theorem \ref{t1.2} via using Theorem \ref{t1.1}, and Lemmas \ref{l3.5},
\ref{l3.6}, and \ref{l3.7}.

\begin{proof}[Proof of Theorem \ref{t1.2}]
We first prove (i). By the assumption $s>(\frac{p_3}{q})'$, we conclude that there
exists an $a\in(1,\frac{p_3}{p_1})$ such that $\frac{a(s-1)}{s}>\frac{q}{p_1}$.
From Theorem \ref{t1.1}, we deduce that, if
$\bz\ge\bz_0$ and $\kappa\in(0,\kappa_0)$, then, for any $\lz\in(0,\fz)$,
\begin{align}\label{3.8}
\omega(E(\bz\lz))\le\bz^{-\frac{a(s-1)}{s}}\omega(E(\lz))+\omega\lf(\lf\{x\in\rn:\
\cm(|f|^{p_2})(x)>(\kappa\lz)^{\frac{p_2}{p_1}}\r\}\r),
\end{align}
where $\bz_0$ and $\kappa_0$ are as in Theorem \ref{t1.1},
which further implies that, for any given $T\in(1,\fz)$,
\begin{align}\label{3.9}
&\int_0^T\lz^{\frac{q}{p_1}-1}\omega(E(\bz\lz))\,d\lz\\
&\quad\le\bz^{-\frac{a(s-1)}{s}}
\int_0^T\lz^{\frac{q}{p_1}-1}\omega(E(\lz))\,d\lz\nonumber \\
&\quad\quad+\int_0^T\lz^{\frac{q}{p_1}-1}
\omega\lf(\lf\{x\in\rn:\ \cm(|f|^{p_2})(x)>(\kappa\lz)^{\frac{p_2}{p_1}}\r\}\r)\,d\lz.\nonumber
\end{align}
By Lemma \ref{l3.5}(i) and a change of variables, we find that
\begin{align}\label{3.10}
&\int_0^\fz\lz^{\frac{q}{p_1}-1}\omega\lf(\lf\{x\in\rn:\ \cm(|f|^{p_2})(x)>(\kappa\lz)^{\frac{p_2}{p_1}}\r\}\r)\,d\lz\\
&\quad=\frac{p_1}{p_2}\kappa^{-\frac{q}{p_1}}\int_0^\fz\lz^{\frac{q}{p_2}-1}
\omega\lf(\lf\{x\in\rn:\ \cm(|f|^{p_2})(x)>\lz\r\}\r)\,d\lz\nonumber\\
&\quad=\frac{p_1}{q}\kappa^{-\frac{q}{p_1}}
\int_\rn\lf[\cm(|f|^{p_2})\r]^{\frac{q}{p_2}}\omega\,dx
\le C\kappa^{-\frac{q}{p_1}}\int_\rn|f|^q v\,dx.\nonumber
\end{align}
Let $\mathrm{I}:=C\kappa^{-\frac{q}{p_1}}\int_\rn|f|^q v\,dx$.
Then, from \eqref{3.9}, \eqref{3.10}, and a change of variables, it follows that
\begin{align*}
\bz^{-\frac{q}{p_1}}\int_0^{\bz T}\lz^{\frac{q}{p_1}-1}\omega(E(\lz))\,d\lz
\le\bz^{-\frac{a(s-1)}{s}}\int_0^T\lz^{\frac{q}{p_1}-1}\omega(E(\lz))\,d\lz+\mathrm{I},
\end{align*}
which, combined with $\bz,\,T>1$, further implies that
$$
\bz^{-\frac{q}{p_1}}\int_0^{T}\lz^{\frac{q}{p_1}-1}\omega(E(\lz))\,d\lz
\le\bz^{-\frac{a(s-1)}{s}}\int_0^T\lz^{\frac{q}{p_1}-1}\omega(E(\lz))\,d\lz+\mathrm{I}.
$$
By this, we know that
\begin{align}\label{3.11}
\bz^{-\frac{q}{p_1}}\lf[1-\bz^{\frac{q}{p_1}-\frac{a(s-1)}s}\r]\int_0^{T}\lz^{\frac{q}{p_1}-1}\omega(E(\lz))\,d\lz
\le C\kappa^{-\frac{q}{p_1}}\int_\rn|f|^q v\,dx.
\end{align}
Using the condition $\frac{a(s-1)}{s}>\frac{q}{p_1}$ and taking $\bz$ large enough
in \eqref{3.11}, we then find that
$$\int_0^{T}\lz^{\frac{q}{p_1}-1}\omega(E(\lz))\,d\lz
\ls\int_\rn|f|^q v\,dx.
$$
Letting $T\to\fz$, we conclude that
\begin{align}\label{3.12}
\int_{\rn}\lf[\cm(|F|^{p_1})\r]^{\frac{q}{p_1}}\omega\,dx\ls\int_\rn|f|^q v\,dx,
\end{align}
which, together with the fact that $\cm(|F|^{p_1})\ge|F|^{p_1}$, implies that
$\|F\|_{L^q_\omega(\rn)}\ls\|f\|_{L^q_v(\rn)}$. This finishes the proof of (i).

Next, we show (ii). We first assume that $t\in(0,\fz)$.
By \eqref{3.8}, we know that, for any given $T\in(1,\fz)$,
\begin{align}\label{3.13}
&\int_0^T\lz^{\frac{t}{p_1}-1}\lf[\omega\lf(\lf\{x\in\rn:\ \cm(|F|^{p_1})(x)>\bz\lz\r\}\r)\r]^{\frac{t}{q}}\,d\lz\\
&\quad\le2^{\frac{t}{q}}\bz^{-\frac{(s-1)at}{sq}}\int_0^T\lz^{\frac{t}{p_1}-1}\lf[\omega\lf(\{x\in\rn:\
\cm(|F|^{p_1})(x)>\lz\}\r)\r]^{\frac{t}{q}}\,d\lz
\nonumber\\
&\quad\quad+2^{\frac{t}{q}}\int_0^T\lz^{\frac{t}{p_1}-1}\lf[\omega\lf(\lf\{x\in\rn:\
\cm(|f|^{p_2})(x)>(\kappa\lz)^{\frac{p_2}{p_1}}\r\}\r)\r]^{\frac{t}{q}}\,d\lz.\nonumber
\end{align}
Let
$$\mathrm{II}:=2^{\frac{t}{q}}\int_0^T\lz^{\frac{t}{p_1}-1}\lf[\omega
\lf(\lf\{x\in\rn:\ \cm(|f|^{p_2})(x)
>(\kappa\lz)^{\frac{p_2}{p_1}}\r\}\r)\r]^{\frac{t}{q}}\,d\lz.
$$
Then we have
\begin{align*}
\mathrm{II}&\le2^{\frac{t}{q}}\frac{p_1}{p_2}\kappa^{-\frac{t}{p_1}}\int_0^\fz\lz^{\frac{t}{p_2}-1}
\lf[\omega\lf(\lf\{x\in\rn:\ \cm(|f|^{p_2})(x)>\lz\r\}\r)\r]^{\frac{t}{q}}\,d\lz\\
&=2^{\frac{t}{q}}\frac{p_1}{q}\kappa^{-\frac{t}{p_1}}\lf\|\cm\lf(|f|^{p_2}\r)\r
\|_{L^{\frac{q}{p_2},\,\frac{t}{p_2}}_\omega(\rn)}^{\frac{t}{p_2}},
\end{align*}
which, combined with \eqref{3.13}, $\bz>1$, and a change of variables, further implies that
\begin{align*}
&\bz^{-\frac{t}{p_1}}\int_0^T\lz^{\frac{t}{p_1}-1}\lf[\omega\lf(\lf\{x\in\rn:\
\cm(|F|^{p_1})(x)>\lz\r\}\r)\r]^{\frac{t}{q}}\,d\lz\\ \nonumber
&\quad\le2^{\frac{t}{q}}\bz^{-\frac{(s-1)at}{sq}}\int_0^T\lz^{\frac{t}{p_1}-1}\lf[\omega\lf(\{x\in\rn:\ \cm(|F|^{p_1})(x)>\lz\}\r)\r]^{\frac{t}{q}}\,d\lz\\ \nonumber
&\quad\quad+2^{\frac{t}{q}}\frac{p_1}{q}\kappa^{-\frac{t}{p_1}}\lf\|\cm\lf(|f|^{p_2}\r)\r\|_{L^{\frac{q}{p_2},\,
\frac{t}{p_2}}_\omega(\rn)}^{\frac{t}{p_2}}.
\end{align*}
From this, we deduce that
\begin{align}\label{3.14}
&\bz^{-\frac{t}{p_1}}\lf[1-2^{\frac{t}{q}}\bz^{\frac{t}{p_1}-\frac{at(s-1)}{sq}}\r]\int_0^{T}\lz^{\frac{t}{p_1}-1}
\lf[\omega\lf(\lf\{x\in\rn:\ \cm(|F|^{p_1})(x)>\lz\r\}\r)\r]^{\frac{t}{q}}\,d\lz\\
&\quad\le2^{\frac{t}{q}}\frac{p_1}{q}\kappa^{-\frac{t}{p_1}}\lf\|\cm\lf(|f|^{p_2}\r)
\r\|_{L^{\frac{q}{p_2},\,\frac{t}{p_2}}_\omega(\rn)}^{\frac{t}{p_2}}.\nonumber
\end{align}
Using the condition $\frac{a(s-1)}{s}>\frac{q}{p_1}$ and taking $\bz$ large enough
in \eqref{3.14}, we then conclude that
$$\int_0^{T}\lz^{\frac{t}{p_1}-1}
\lf[\omega\lf(\lf\{x\in\rn:\ \cm(|F|^{p_1})(x)>\lz\r\}\r)\r]^{\frac{t}{q}}\,d\lz
\ls\lf\|\cm\lf(|f|^{p_2}\r)\r\|_{L^{\frac{q}{p_2},\,\frac{t}{p_2}}_\omega(\rn)}^{\frac{t}{p_2}}.
$$
Letting $T\to\fz$, we find that
\begin{align*}
\lf\|\cm\lf(|F|^{p_1}\r)\r\|_{L^{\frac{q}{p_1},\,\frac{t}{p_1}}_\omega(\rn)}^{\frac{t}{p_1}}
\ls\lf\|\cm\lf(|f|^{p_2}\r)\r\|_{L^{\frac{q}{p_2},\,\frac{t}{p_2}}_\omega(\rn)}^{\frac{t}{p_2}},
\end{align*}
which, together with the fact that $\cm(|F|^{p_1})\ge |F|^{p_1}$, implies that
$$
\lf\||F|^{p_1}\r\|_{L^{\frac{q}{p_1},\,\frac{t}{p_1}}_\omega(\rn)}^{\frac{t}{p_1}}
\ls\lf\|\cm\lf(|f|^{p_2}\r)\r\|_{L^{\frac{q}{p_2},\,\frac{t}{p_2}}_\omega(\rn)}^{\frac{t}{p_2}}.
$$
From this, Remark \ref{r1.2}, and Lemma \ref{l3.5}(ii), it follows that
\begin{align*}
\|F\|^{t}_{L^{q,\,t}_\omega(\rn)}&=\lf\||F|^{p_1}\r\|_{L^{\frac{q}{p_1},\,\frac{t}{p_1}}_\omega(\rn)}^{\frac{t}{p_1}}
\ls\lf\|\cm\lf(|f|^{p_2}\r)\r\|_{L^{\frac{q}{p_2},\,\frac{t}{p_2}}_\omega(\rn)}^{\frac{t}{p_2}}\\
&\ls\lf\||f|^{p_2}\r\|_{L^{\frac{q}{p_2},\,\frac{t}{p_2}}_v(\rn)}^{\frac{t}{p_2}}\sim
\lf\|f\r\|_{L^{q,\,t}_v(\rn)}^{t}.
\end{align*}
Thus, $\|F\|_{L^{q,\,t}_\omega(\rn)}\ls\|f\|_{L^{q,\,t}_v(\rn)}$,
which completes the proof of (ii) in the case of $t\in(0,\fz)$.

Now, we consider the case of $t=\fz$. By \eqref{3.8}, we know that, if
$\bz\ge\bz_0$ and $\kappa\in(0,\kappa_0)$, then, for any $\lz\in(0,\fz)$,
\begin{align}\label{3.15}
\lz\lf[\omega(E(\bz\lz))\r]^{\frac{p_1}{q}}\le\bz^{-\frac{ap_1(s-1)}{sq}}\lz\lf[\omega(E(\lz))\r]^{\frac{p_1}{q}}
+\lz\lf[\omega\lf(\lf\{x\in\rn:\ \cm(|f|^{p_2})(x)>(\kappa\lz)^{\frac{p_2}{p_1}}\r\}\r)\r]^{\frac{p_1}{q}}.
\end{align}
From the definition of the space $L^{\frac{q}{p_2},\,\fz}_\omega(\rn)$, we deduce that
\begin{align*}
&\sup_{\lz\in(0,\fz)}\lf\{\lz\lf[\omega\lf(\lf\{x\in\rn:\ \cm(|f|^{p_2})(x)>(\kappa\lz)^{\frac{p_2}{p_1}}\r\}\r)
\r]^{\frac{p_1}{q}}\r\}\\ \nonumber
&\quad=\frac{1}{\kappa}\sup_{\lz\in(0,\fz)}\lf\{(\kappa\lz)^{\frac{p_2}{p_1}}\lf[\omega\lf(\lf\{x\in\rn:\ \cm(|f|^{p_2})(x)>(\kappa\lz)^{\frac{p_2}{p_1}}\r\}\r)\r]^{\frac{p_2}{q}}\r\}^{\frac{p_1}{p_2}}\\ \nonumber
&\quad=\frac{1}{\kappa}\lf\|\cm(|f|^{p_2})\r\|^{\frac{p_1}{p_2}}_{L^{\frac{q}{p_2},\,\fz}_\omega(\rn)},
\end{align*}
which, combined with \eqref{3.15}, further implies that
\begin{align}\label{3.16}
\bz^{-1}\lf\|\cm(|F|^{p_1})\r\|_{L^{\frac{q}{p_1},\,\fz}_\omega(\rn)}
\le\bz^{-\frac{ap_1(s-1)}{sq}}\lf\|\cm(|F|^{p_1})\r\|_{L^{\frac{q}{p_1},\,\fz}_\omega(\rn)}
+\kappa^{-1}\lf\|\cm(|f|^{p_2})\r\|^{\frac{p_1}{p_2}}_{L^{\frac{q}{p_2},\,\fz}_\omega(\rn)}.
\end{align}
Using the condition $\frac{a(s-1)}{s}>\frac{q}{p_1}$ and taking $\bz$ large enough
in \eqref{3.16}, we then find that
\begin{align*}
\lf\|\cm(|F|^{p_1})\r\|_{L^{\frac{q}{p_1},\,\fz}_\omega(\rn)}
\ls\lf\|\cm(|f|^{p_2})\r\|^{\frac{p_1}{p_2}}_{L^{\frac{q}{p_2},\,\fz}_\omega(\rn)},
\end{align*}
which, together with the fact that $\cm(|F|^{p_1})\ge|F|^{p_1}$, Lemma \ref{l3.5}(ii),
and Remark \ref{r1.2}, further implies that $\|F\|_{L^{q,\,\fz}_\omega(\rn)}\ls\|f\|_{L^{q,\,\fz}_v(\rn)}$.
This finishes the proof of the case $t=\fz$ and hence of (ii).

Next, we prove (iii). To this end, we borrow some ideas from the proof of \cite[Theorem 2.3]{mp12}.
By the assumption $\tz\in(\frac{nq}{p_3},n]$, we know that there exists a $\rho_0\in(0,\tz)$
such that $\tz-\rho_0>\frac{nq}{p_3}$.
Let $x\in\rn$ and $r\in(0,\fz)$. For any $y\in\rn$, let
$\omega_x(y):=\min\{|x-y|^{-n+\tz-\rho_0}, r^{-n+\tz-\rho_0}\}$.
From Lemma \ref{l3.6}, we deduce that $\omega_x\in A_p(\rn)$ for any $p\in(1,\fz)$.
Moreover, by Lemmas \ref{l3.1}(iii), we further conclude that
$\omega_x\in RH_\gz(\rn)$ with any $\gz\in(1,\frac{n}{n+\rho_0-\tz})$, which, combined with $\tz-\rho_0>\frac{nq}{p_3}$
and $(\frac{n}{n+\rho_0-\tz})'=\frac{n}{\tz-\rho_0}$,
implies that $\frac{n}{n+\rho_0-\tz}>(\frac{p_3}{q})'$. From this, it follows that
there exists an $s\in((\frac{p_3}{q})',\frac{n}{n+\rho_0-\tz})$ such that $\omega_x\in RH_s(\rn)$.

Thus, applying the conclusion of (ii) in the special case that $\omega=v:=\omega_x$
and $\Phi(t):=t^{(\frac{q}{p_2})'s_0}$ with some $s_0\in(1,\fz)$,
we find that, for any $t\in[0,\fz)$,
\begin{align}\label{3.17}
\|F\|_{L^{q,\,t}(B(x,r))}=r^{\frac{n-\tz+\rho_0}{q}}\|F\|_{L^{q,\,t}_{\omega_x}(B(x,r))}
\ls r^{\frac{n-\tz+\rho_0}{q}}\|f\|_{L^{q,\,t}_{\omega_x}(\rn)}.
\end{align}
For any $\lz\in(0,\fz)$, let $E_\lz:=\{x\in\rn:\ |f(x)|>\lz\}$.
Then, by the fact that $\omega_x\le r^{-n+\tz-\rho_0}$ in $\rn$, we know that
\begin{align}\label{3.18}
\omega_x(E_\lz)=\int_{E_\lz}\omega_x\,dy\le\int_0^{r^{-n+\tz-\rho_0}}
\lf|E_\lz\cap B\lf(x,\az^{\frac{1}{-n+\tz-\rho_0}}\r)\r|\,d\az.
\end{align}

To finish the proof of (iii), we consider the following three cases
on the index $t$.

\emph{Case 1.} $t=\fz$. In this case, from the definition of
the space $L^{q,\,\fz;\,\tz}(\rn)$ and \eqref{3.18}, we deduce that
\begin{align*}
\|f\|_{L^{q,\,\fz}_{\omega_x}(\rn)}&=\sup_{\lz\in(0,\fz)}\lf\{\lz\lf[\omega_x(E_\lz)\r]^{\frac1q}\r\}
\le\sup_{\lz\in(0,\fz)}\lf\{\lz\lf[\int_0^{r^{-n+\tz-\rho_0}}
\lf|E_\lz\cap B\lf(x,\az^{\frac{1}{-n+\tz-\rho_0}}\r)\r|\,d\az\r]^{\frac1q}\r\}\\ \nonumber
&\le\|f\|_{L^{q,\,\fz;\,\tz}(\rn)}\lf[\int_0^{r^{-n+\tz-\rho_0}}\az^{\frac{n-\tz}{-n+\tz-\rho_0}}\,d\az\r]^{\frac{1}{q}}
\ls r^{-\frac{\rho_0}{q}}\|f\|_{L^{q,\,\fz;\,\tz}(\rn)},
\end{align*}
which, together with \eqref{3.17}, implies that $r^{\frac{\tz-n}{q}}\|F\|_{L^{q,\,\fz}(B(x,r))}\ls\|f\|_{L^{q,\,\fz;\,\tz}(\rn)}$.
By this and the arbitrariness of $x\in\rn$ and $r\in(0,\fz)$, we conclude that
$\|F\|_{L^{q,\,\fz;\,\tz}(\rn)}\ls\|f\|_{L^{q,\,\fz;\,\tz}(\rn)}$.

\emph{Case 2.} $t\in(q,\fz)$. In this case, from \eqref{3.18} and the Minkowski inequality, it follows that
\begin{align}\label{3.19}
\|f\|^t_{L^{q,\,t}_{\omega_x}(\rn)}&=q\int_0^\fz\lf[\lz^q\int_{E_\lz}\omega_x\,
dy\r]^{\frac tq}\frac{d\lz}{\lz}\\
&\le q\lf\{\int_0^{r^{-n+\tz-\rho_0}}\lf[\int_0^\fz\lf[\lz^q
\lf|E_\lz\cap B\lf(x,\az^{\frac{1}{-n+\tz-\rho_0}}\r)\r|\r]^{\frac{t}{q}}
\frac{d\lz}{\lz}\r]^{\frac{q}{t}}\,d\az\r\}^{\frac{t}{q}}\nonumber\\
&\le q \|f\|^t_{L^{q,\,t;\,\tz}(\rn)}\lf[\int_0^{r^{-n+\tz-\rho_0}}
\az^{\frac{n-\tz}{-n+\tz-\rho_0}}\,d\az\r]^{\frac{t}{q}}
\ls r^{-\frac{\rho_0 t}{q}}\|f\|^t_{L^{q,\,t;\,\tz}(\rn)}.\nonumber
\end{align}
By \eqref{3.19} and \eqref{3.17}, we find that
$$\|F\|^t_{L^{q,\,t}(B(x,r))}\ls  r^{\frac{(n-\tz+\rho_0)t}{q}}r^{-\frac{\rho_0 t}{q}}\|f\|^t_{L^{q,\,t;\,\tz}(\rn)}
\sim r^{\frac{(n-\tz)t}{q}}\|f\|^t_{L^{q,\,t;\,\tz}(\rn)},
$$
which implies that $r^{\frac{\tz-n}{q}}\|F\|_{L^{q,\,t}(B(x,r))}\ls\|f\|_{L^{q,\,t;\,\tz}(\rn)}$.
From this and the arbitrariness of $x\in\rn$ and $r\in(0,\fz)$, we further deduce that
$\|F\|_{L^{q,\,t;\,\tz}(\rn)}\ls\|f\|_{L^{q,\,t;\,\tz}(\rn)}$ in the case of $t\in(q,\fz)$.

\emph{Case 3}. $t\in(0,q]$. In this case, by the fact that
$(\sum_{j=1}^\fz a_j)^{\frac{t}{q}}\le\sum_{j=1}^\fz a_j^{\frac{t}{q}}$
holds true for any non-negative sequence $\{a_j\}_{j=1}^\fz$, we conclude that
\begin{align*}
&\lf[\int_0^{r^{-n+\tz-\rho_0}}\lf|E_\lz\cap B\lf(x,\az^{\frac{1}{-n+\tz-\rho_0}}\r)\r|\,d\az\r]^{\frac{t}{q}}\\
&\quad=\lf[\sum_{j=1}^\fz\int_{2^{-j}r^{-n+\tz-\rho_0}}^{2^{-j+1}r^{-n+\tz-\rho_0}}
\lf|E_\lz\cap B\lf(x,\az^{\frac{1}{-n+\tz-\rho_0}}\r)\r|\,d\az\r]^{\frac{t}{q}}\\
&\quad\le\lf[\sum_{j=1}^\fz 2^{-j}r^{-n+\tz-\rho_0}
\lf|E_\lz\cap B\lf(x,2^{\frac{-j+1}{-n+\tz-\rho_0}}r\r)\r|\r]^{\frac{t}{q}}\\
&\quad\lesssim\int_0^{2r^{-n+\tz-\rho_0}}\lf[\az\lf|E_\lz\cap B\lf(x,
\az^{\frac{1}{-n+\tz-\rho_0}}\r)\r|\r]^{\frac tq}\frac{d\az}{\az},
\end{align*}
which, combined with \eqref{3.18} and the Fubini theorem, further implies that
\begin{align*}
\|f\|^t_{L^{q,\,t}_{\omega_x}(\rn)}&=q\int_0^\fz\lf[\lz^q\int_{E_\lz}\omega_x\,
dy\r]^{\frac tq}\frac{d\lz}{\lz}\\ \nonumber
&\lesssim\int_0^\fz\lz^t\int_0^{2r^{-n+\tz-\rho_0}}
\lf[\az\lf|E_\lz\cap B\lf(x,\az^{\frac{1}{-n+\tz-\rho_0}}\r)\r|\r]^{\frac tq}\frac{d\az}{\az}\frac{d\lz}{\lz}\\ \nonumber
&\lesssim \|f\|^t_{L^{q,\,t;\,\tz}(\rn)}\int_0^{2r^{-n+\tz-\rho_0}}\az^{\frac{t}{q}}\az^{\frac{(n-\tz)t}{(-n+\tz-\rho_0)q}}
\frac{d\az}{\az}
\ls r^{-\frac{\rho_0 t}{q}}\|f\|^t_{L^{q,\,t;\,\tz}(\rn)}.
\end{align*}
From this and \eqref{3.17}, it follows that
$$\|F\|^t_{L^{q,\,t}(B(x,r))}\ls  r^{\frac{(n-\tz+\rho_0)t}{q}}r^{-\frac{\rho_0 t}{q}}\|f\|^t_{L^{q,\,t;\,\tz}(\rn)}
\sim r^{\frac{(n-\tz)t}{q}}\|f\|^t_{L^{q,\,t;\,\tz}(\rn)}
$$
and hence $\|F\|_{L^{q,\,t;\,\tz}(\rn)}\ls\|f\|_{L^{q,\,t;\,\tz}(\rn)}$. This finishes the proof
of the case $t\in(0,q]$ and hence of (iii).

Finally, we prove (iv). Let $q_0\in(\max\{p_1,\,p_2\},p_3)$ and $\omega_0\in A_{\frac{q_0}{p_2}}(\rn)\cap RH_s(\rn)$
with some $s\in((\frac{p_3}{q_0})',\fz]$. Take $\epsilon_0\in(0,p_3-q_0)$ and $q(\cdot):=\frac{q_0+\epsilon_0}{p_2 q_0}p(\cdot)$.
Then $\omega_0\in A_{\frac{(q_0+\epsilon_0)/p_2}{(q_0+\epsilon_0)/q_0}}(\rn)$, $\omega_0\in RH_{(\frac{p_3(q_0+\epsilon_0)/(p_2 q_0)}{(q_0+\epsilon_0)/q_0})'}(\rn)$, and
$$1<\frac{q_0+\epsilon_0}{q_0}<q_-\le q_+<\frac{p_3(q_0+\epsilon_0)}{p_2q_0},$$
where $q_-$ and $q_+$ are as in \eqref{1.7} with $p(\cdot)$ replaced by $q(\cdot)$.

Moreover, by the conclusion of (i) in the special case that $\omega=v:=\omega_0$,
we know that $\|F\|_{L^{q_0}_{\omega_0}(\rn)}\ls\|f\|_{L^{q_0}_{\omega_0}(\rn)}$,
which further implies that
$$\lf\||F|^{\frac{p_2q_0}{q_0+\epsilon_0}}\r\|_{L^{\frac{q_0+\epsilon_0}{p_2}}_{\omega_0}(\rn)}\ls
\lf\||f|^{\frac{p_2q_0}{q_0+\epsilon_0}}\r\|_{L^{\frac{q_0+\epsilon_0}{p_2}}_{\omega_0}(\rn)}.$$
From this and Lemma \ref{l3.7}, we deduce that
$$\lf\||F|^{\frac{p_2q_0}{q_0+\epsilon_0}}\r\|_{L^{q(\cdot)}(\rn)}\ls
\lf\||f|^{\frac{p_2q_0}{q_0+\epsilon_0}}\r\|_{L^{q(\cdot)}(\rn)},
$$
which, together with $q(\cdot)=\frac{q_0+\epsilon_0}{p_2 q_0}p(\cdot)$ and Remark \ref{r1.3}, implies that
$\|F\|_{L^{p(\cdot)}(\rn)}\ls\|f\|_{L^{p(\cdot)}(\rn)}$.
This finishes the proof of (iv) and hence of Theorem \ref{t1.2}.
\end{proof}

To show Theorem \ref{t1.3} via using Theorem \ref{t1.1}, we need the following two-weight
boundedness of the fractional Hardy--Littlewood maximal operator on weighted Lebesgue spaces
(see, for instance, \cite[Theorem 2.11]{p94} and \cite[Theorem 5.37]{cmp11}).

\begin{lemma}\label{l3.8}
Let $\gz\in(0,1)$ and $\cm_\gz$ be the fractional Hardy--Littlewood maximal operator on $\rn$,
$p\in(1,\gz^{-1})$, $q\in(1,\fz)$ satisfy $\frac{1}{q}=\frac{1}{p}-\gz$,
and $\Phi$ be a doubling Young function
satisfying
$$\int_c^\fz\lf[\frac{t^{p'}}{\Phi(t)}\r]^{p-1}\frac{dt}{t}<\fz
$$
for some constant $c\in(0,\fz)$. Assume that the weights $\omega$ and $v$ satisfy
$$\sup_{B\subset\rn}\lf(\fint_B \omega\,dx\r)\lf\|v^{-\frac{1}{p}}\r\|^{q}_{\Phi,\,B}<\fz,
$$
where the supremum is taken over all balls $B$ of $\rn$.
Then $\cm_\gz$ is bounded from $L^p_v(\rn)$ to $L^q_\omega(\rn)$.
\end{lemma}

Now, we prove Theorem \ref{t1.3} by using Theorem \ref{t1.1} and Lemma \ref{l3.8}.

\begin{proof}[Proof of Theorem \ref{t1.3}]
We first show (i). By the assumption $s>(\frac{p_3}{q})'$, we conclude that there
exists an $a\in(1,\frac{p_3}{p_1})$ such that $\frac{a(s-1)}{s}>\frac{q}{p_1}$.
From Theorem \ref{t1.1}, it follows that, if
$\bz\ge\bz_0$ and $\kappa\in(0,\kappa_0)$, then, for any $\lz\in(0,\fz)$,
\begin{align*}
\omega(E(\bz\lz))\le\bz^{-\frac{a(s-1)}{s}}\omega(E(\lz))+\omega\lf(\lf\{x\in\rn:\
\cm_\gz(|f|^{p_2})(x)>(\kappa\lz)^{\frac{p_2}{p_1}}\r\}\r),
\end{align*}
where $\bz_0$ and $\kappa_0$ are as in Theorem \ref{t1.1},
which further implies that, for any given $T\in(1,\fz)$,
\begin{align*}
\int_0^T\lz^{\frac{q}{p_1}-1}\omega(E(\bz\lz))\,d\lz&\le\bz^{-\frac{a(s-1)}{s}}
\int_0^T\lz^{\frac{q}{p_1}-1}\omega(E(\lz))\,d\lz\\ \nonumber
&\quad+\int_0^T\lz^{\frac{q}{p_1}-1}
\omega\lf(\lf\{x\in\rn:\ \cm_\gz(|f|^{p_2})(x)>(\kappa\lz)^{\frac{p_2}{p_1}}\r\}\r)\,d\lz.
\end{align*}
Similarly to the proof of \eqref{3.10}, by Lemma \ref{l3.8}, we know that
\begin{align}\label{3.20}
\int_0^\fz\lz^{\frac{q}{p_1}-1}\omega\lf(\lf\{x\in\rn:\ \cm_\gz(|f|^{p_2})(x)>(\kappa\lz)^{\frac{p_2}{p_1}}\r\}\r)\,d\lz\ls \kappa^{-\frac{q}{p_1}}\lf(\int_\rn|f|^p v\,dx\r)^{\frac{q}{p}}.
\end{align}
Using \eqref{3.20} and repeating the proof of \eqref{3.12}, we conclude that
\begin{align*}
\int_{\rn}\lf[\cm(|F|^{p_1})\r]^{\frac{q}{p_1}}\omega\,dx\ls\lf(\int_\rn|f|^p v\,dx\r)^{\frac{q}{p}},
\end{align*}
which, combined with the fact that $\cm(|F|^{p_1})\ge|F|^{p_1}$, further implies that $\|F\|_{L^q_\omega(\rn)}\ls\|f\|_{L^p_v(\rn)}$.
This finishes the proof of (i).

Finally, we prove (ii). From the assumption $\tz\in(\frac{nq}{p_3},n]$, it follows that there exists a $\rho_0\in(0,\tz)$
such that $\tz-\rho_0>\frac{nq}{p_3}$. Let $r\in(0,\fz)$ and $x\in\rn$. For any $y\in\rn$, let
$\omega_x(y):=\min\{|x-y|^{-n+\tz-\rho_0}, r^{-n+\tz-\rho_0}\}$ and
$v_x(y):=\min\{|x-y|^{-n+\wz{\tz}-\frac{p\rho_0}{q}}, r^{-n+\wz{\tz}-\frac{p\rho_0}{q}}\}$, where $\wz{\tz}:=n-\frac{p}{q}(n-\tz)$.
By Lemmas \ref{l3.1}(iii) and \ref{l3.6}, we conclude that
$\omega_x\in RH_\gz(\rn)$ with any $\gz\in(1,\frac{n}{n+\rho_0-\tz})$, which, together with $\tz-\rho_0>\frac{nq}{p_3}$ and $(\frac{n}{n+\rho_0-\tz})'=\frac{n}{\tz-\rho_0}$,
implies that $\frac{n}{n+\rho_0-\tz}>(\frac{p_3}{q})'$.
From this, we deduce that there exists an $s\in((\frac{p_3}{q})',\frac{n}{n+\rho_0-\tz})$ such that $\omega_x\in RH_s(\rn)$.
Moreover, it is easy to show that $\omega_x$ and $v_x$ satisfy \eqref{1.9}
with $\Phi(t):=t^{(\frac{p}{p_2})'s_0}$ for some $s_0\in(1,\fz)$ and any $t\in[0,\fz)$.

Thus, applying the conclusion of (i), we find that
\begin{align}\label{3.21}
\|F\|_{L^{q}(B(x,r))}=r^{\frac{n-\tz+\rho_0}{q}}\|F\|_{L^{q}_{\omega_x}(B(x,r))}
\ls r^{\frac{n-\wz{\tz}+\frac{p\rho_0}{q}}{p}}\|f\|_{L^{p}_{v_x}(\rn)}.
\end{align}
Furthermore, similarly to the proof of \eqref{3.19}, we conclude that
\begin{align*}
\|f\|_{L^{p}_{v_x}(\rn)}
\ls r^{-\frac{\rho_0}{q}}\|f\|_{\cm^{\wz{\tz}}_p(\rn)},
\end{align*}
which, combined with \eqref{3.21}, $\wz{\tz}:=n-\frac{p}{q}(n-\tz)$,
and the arbitrariness of $r\in(0,\fz)$ and $x\in\rn$,
further implies that $\|F\|_{\cm^{\tz}_q(\rn)}\ls\|f\|_{\cm^{\wz{\tz}}_p(\rn)}$.
This finishes the proof of (ii) and hence of Theorem \ref{t1.3}.
\end{proof}

\section{Proofs of Theorems \ref{t2.1}, \ref{t2.2}, \ref{t2.3},
and \ref{t2.4}\label{s4}}

In this section, we prove Theorems \ref{t2.1}, \ref{t2.2}, \ref{t2.3},
and \ref{t2.4} by using Theorems \ref{t1.2} and \ref{t1.3}.

To show Theorem \ref{t2.1}, we need the boundedness of the Calder\'on--Zygmund operator
from $L^1(\rn)$ to $L^{1,\,\fz}(\rn)$, and the Kolmogorov inequality.

\begin{lemma}\label{l4.1}
Let $T$ be a Calder\'on--Zygmund operator as in Definition \ref{d2.1}. Then $T$ is bounded from $L^1(\rn)$
to $L^{1,\,\fz}(\rn)$.
\end{lemma}

Lemma \ref{l4.1} is well known (see, for instance, \cite{g14,St93}).

\begin{lemma}\label{l4.2}
Let $S$ be a bounded operator from $L^1(\rn)$ to $L^{1,\,\fz}(\rn)$, $\nu\in(0,1)$ a constant, and
the measurable subset $E$ of $\rn$ satisfy $|E|<\fz$.
Then there exists a positive constant $C$, depending only on $\nu$, such that, for any $f\in L^1(\rn)$,
$$\int_E|S(f)|^{\nu}\,dx\le C|E|^{1-\nu}\|f\|^\nu_{L^1(\rn)}.
$$
\end{lemma}

The conclusion of Lemma \ref{l4.2}, called the Kolmogorov inequality,
is well known (see, for instance, \cite[Lemma 5.16]{d01} and \cite[p.\,100]{g14}).

We now prove Theorem \ref{t2.1} via using Theorem \ref{t1.2}, and Lemmas \ref{l4.1} and \ref{l4.2}.
In what follows, let $L^\fz_{\rm c}(\rn)$ denote the set of all \emph{bounded measurable functions
on $\rn$ with compact support}.

\begin{proof}[Proof of Theorem \ref{t2.1}]
Let $f\in L^\fz_{\rm c}(\rn)$ and $F:=T(f)$. For any ball $B:=B(x_B,r_B)$ of $\rn$ with $x_B\in \rn$ and $r_B\in(0,\fz)$,
let $F_B:=T(f\mathbf{1}_{8B})$ and $R_B:=T(f-f\mathbf{1}_{8B})$. Then $|F|\le|F_B|+|R_B|$. Let $q\in(1,\fz)$ and $\nu\in(0,1)$.
By the assumption $\omega\in A_\fz(\rn)$ and Lemma \ref{l3.1}(i), we find that there exists an $s\in(1,\fz]$
such that $\omega\in RH_s(\rn)$.

From Lemmas \ref{l4.1} and \ref{l4.2}, we deduce that, for any $x_1\in B$,
\begin{align}\label{4.1}
\lf(\fint_B|F_B|^{\nu}\,dx\r)^{\frac1{\nu}}\ls\fint_{8B}|f|\,dx
\ls\cm(f)(x_1).
\end{align}
Moreover,  by $\supp(f-f\mathbf{1}_{8B})\subset\rn\backslash(8B)$, we conclude that, for any $x,\,y,\,x_1\in B$,
\begin{align*}
&\lf|T\lf(f\mathbf{1}_{\rn\backslash(8B)}\r)(x)-T\lf(f\mathbf{1}_{\rn\backslash(8B)}\r)(y)\r|\\
&\quad\le\int_{\rn\backslash(8B)}
|K(x,z)-K(y,z)||f(z)|\,dz\\
&\quad\ls|x-y|^{\dz}\int_{\rn\backslash(8B)}\frac{|f(z)|}{|x-z|^{n+\dz}}\,dz\\
&\quad\ls r_B^{\dz}\sum_{j=2}^\fz\frac{1}{(2^j r_B)^{n+\dz}}\int_{(2^{j+1}B)\backslash(2^jB)}|f(z)|\,dz
\ls\cm(f)(x_1),
\end{align*}
which further implies that, for any $x_1\in B$,
\begin{align}\label{4.2}
\lf\|T(f\mathbf{1}_{\rn\backslash(8B)})\r\|_{L^\fz(B)}&\ls\inf_{y\in B}\lf\{\lf|T(f\mathbf{1}_{\rn\backslash(8B)})(y)\r|\r\}
+\cm(f)(x_1)\\
&\ls\lf[\fint_B|T(f\mathbf{1}_{\rn\backslash(8B)})|^\nu\,dy\r]^{\frac{1}{\nu}}
+\cm(f)(x_1)\nonumber\\
&\sim\lf[\fint_B|R_B|^\nu\,dy\r]^{\frac{1}{\nu}}
+\cm(f)(x_1).\nonumber
\end{align}
From \eqref{4.2} and \eqref{4.1}, it follows that, for any $x_1,\,x_2\in B$,
\begin{align}\label{4.3}
\lf\|R_B\r\|_{L^\fz(B)}
&\ls\lf(\fint_B|R_B|^\nu\,dy\r)^{\frac{1}{\nu}}+\cm(f)(x_1)\\
&\ls\lf(\fint_B|F|^\nu\,dy\r)^{\frac{1}{\nu}}+\lf(\fint_B|F_B|^\nu\,dy\r)^{\frac{1}{\nu}}+\cm(f)(x_1)\nonumber\\
&\ls\lf(\fint_B|F|^\nu\,dy\r)^{\frac{1}{\nu}}+\cm(f)(x_1)
\ls\lf[\cm(|F|^\nu)(x_2)\r]^{\frac{1}{\nu}}+\cm(f)(x_1).\nonumber
\end{align}
By \eqref{4.1} and \eqref{4.3}, we find that \eqref{1.4} and \eqref{1.5} hold true
for $p_1:=\nu$, $p_2:=1$, $p_3:=\fz$, and $\epsilon:=0$. Thus, applying Theorem \ref{t1.2}, we then complete
the proof of Theorem \ref{t2.1}.
\end{proof}

To show Theorem \ref{t2.2}, we need the following boundedness of the Littlewood--Paley $g$-function from $L^1(\rn)$
to $L^{1,\,\fz}(\rn)$ (see, for instance, \cite{s58,xpy15}).

\begin{lemma}\label{l4.3}
For any $f\in L^1(\rn)$, $g(f)\in L^{1,\,\fz}(\rn)$, and there exists a positive constant $C$, independent of $f$,
such that $\|g(f)\|_{L^{1,\,\fz}(\rn)}\le C\|f\|_{L^1(\rn)}$.
\end{lemma}

Next, we prove Theorem \ref{t2.2} via using Theorem \ref{t1.2} and Lemmas \ref{l4.2} and \ref{l4.3}.

\begin{proof}[Proof of Theorem \ref{t2.2}]
Let $f\in L^\fz_{\rm c}(\rn)$ and $F:=g(f)$. For any ball $B:=B(x_B,r_B)$ of $\rn$ with $x_B\in \rn$ and $r_B\in(0,\fz)$, let $F_B:=g(f\mathbf{1}_{8B})$
and $R_B:=g(f-f\mathbf{1}_{8B})$. Then $|F|\le|F_B|+|R_B|$. Let $q\in(1,\fz)$ and $\nu\in(0,1)$.
By the assumption $\omega\in A_\fz(\rn)$ and Lemma \ref{l3.1}(i), we conclude that there exists an $s\in(1,\fz]$
such that $\omega\in RH_s(\rn)$.

From Lemmas \ref{l4.3} and \ref{l4.2}, it follows that, for any $x_1\in B$,
\begin{align}\label{4.4}
\lf(\fint_B|F_B|^{\nu}\,dx\r)^{\frac1{\nu}}\ls\fint_{8B}|f|\,dx
\ls\cm(f)(x_1).
\end{align}
Moreover, by the definition of $g(f)$, we know that, for any $x,\,y\in B$,
\begin{align}\label{4.5}
|R_B(x)|\le\lf[\int_0^\fz\lf|\phi_t\ast(f\mathbf{1}_{\rn\backslash(8B)})(x)-
\phi_t\ast(f\mathbf{1}_{\rn\backslash(8B)})(y)\r|^2\frac{dt}{t}\r]^{1/2}+g\lf(f\mathbf{1}_{\rn\backslash(8B)}\r)(y).
\end{align}
From the assumption $\phi\in\cs(\rn)$ and the mean value theorem, we deduce that, for any given $\varepsilon\in(1,\fz)$,
and for any $t\in(0,\fz)$ and $x,\,y\in B$,
\begin{align*}
&\lf|\phi_t\ast(f\mathbf{1}_{\rn\backslash(8B)})(x)-
\phi_t\ast(f\mathbf{1}_{\rn\backslash(8B)})(y)\r|\\
&\quad=\lf|\frac{1}{t^n}\int_{\rn\backslash(8B)}
\lf[\phi\lf(\frac{x-z}{t}\r)-\phi\lf(\frac{y-z}{t}\r)\r]f(z)\,dz\r|\\ \nonumber
&\quad\ls\frac{1}{t^n}\int_{\rn\backslash(8B)}\lf|\nabla\phi\lf(\frac{\tz(x-y)-z}{t}\r)\r|\lf|\frac{x-y}{t}\r||f(z)|\,dz\\ \nonumber
&\quad\ls t^{\varepsilon-1}r_B\int_{\rn\backslash(8B)}\frac{|f(z)|}{(t+|\tz(x-y)-z|)^{n+\varepsilon}}\,dz
\ls t^{\varepsilon-1}r_B\sum_{j=2}^\fz\int_{2^jB}\frac{|f(z)|}{(t+2^jr_B)^{n+\varepsilon}}\,dz,
\end{align*}
where $\tz\in(0,1)$ is a constant, which, together with the Minkowski inequality,
further implies that, for any $x,\,y,\,x_1\in B$,
\begin{align}\label{4.6}
&\lf[\int_0^\fz\lf|\phi_t\ast(f\mathbf{1}_{\rn\backslash(8B)})(x)-
\phi_t\ast(f\mathbf{1}_{\rn\backslash(8B)})(y)\r|^2\frac{dt}{t}\r]^{\frac12}\\
&\quad\ls\lf\{\int_0^\fz\lf[t^{\varepsilon-1}r_B\sum_{j=2}^\fz\int_{2^jB}\frac{|f(z)|}
{(t+2^jr_B)^{n+\varepsilon}}\,dz\r]^2\frac{dt}{t}\r\}^{\frac12}
\nonumber\\
&\quad\ls\sum_{j=2}^\fz\lf\{\int_0^\fz\lf[t^{\varepsilon-1}r_B\int_{2^jB}\frac{|f(z)|}
{(t+2^jr_B)^{n+\varepsilon}}\,dz\r]^2\frac{dt}{t}\r\}^{\frac12}
\nonumber\\
&\quad\ls\sum_{j=2}^\fz\lf\{\int_0^{2^jr_B}\lf[t^{\varepsilon-1}r_B\int_{2^jB}\frac{|f(z)|}
{(t+2^jr_B)^{n+\varepsilon}}\,dz\r]^2\frac{dt}{t}\r\}^{\frac12}
+\sum_{j=2}^\fz\lf\{\int_{2^jr_B}^\fz\cdots\frac{dt}{t}\r\}^{\frac12}\nonumber\\
&\quad\ls\sum_{j=2}^\fz\lf\{\int_0^{2^jr_B}t^{2(\varepsilon-1)}\frac{dt}{t}\r\}^{\frac12}
2^{-j\varepsilon}r_B^{1-\varepsilon}\cm(f)(x_1)
+\sum_{j=2}^\fz\lf\{\int_{2^jr_B}^\fz t^{-2}\frac{dt}{t}\r\}^{\frac12}r_B\cm(f)(x_1)\nonumber\\
&\quad\ls\sum_{j=2}^\fz2^{-j}\cm(f)(x_1)\ls\cm(f)(x_1).\nonumber
\end{align}
By \eqref{4.5} and \eqref{4.6}, we conclude that, for any $x,\,x_1\in B$,
$$
|R_B(x)|\ls\inf_{y\in B}\lf\{g\lf(f\mathbf{1}_{\rn\backslash(8B)}\r)(y)\r\}+\cm(f)(x_1)
\ls\lf(\fint_B|R_B|^\nu\,dy\r)^{\frac{1}{\nu}}+\cm(f)(x_1),
$$
which, combined with \eqref{4.4}, implies that, for any $x_1,\,x_2\in B$,
\begin{align}\label{4.7}
\lf\|R_B\r\|_{L^\fz(B)}
&\ls\lf(\fint_B|R_B|^\nu\,dy\r)^{\frac{1}{\nu}}+\cm(f)(x_1)\\
&\ls\lf(\fint_B|F|^\nu\,dy\r)^{\frac{1}{\nu}}+\lf(\fint_B|F_B|^\nu\,dy\r)^{\frac{1}{\nu}}+\cm(f)(x_1)\nonumber\\
&\ls\lf(\fint_B|F|^\nu\,dy\r)^{\frac{1}{\nu}}+\cm(f)(x_1)
\ls\lf[\cm(|F|^\nu)(x_2)\r]^{\frac{1}{\nu}}+\cm(f)(x_1).\nonumber
\end{align}
Then, from \eqref{4.4} and \eqref{4.7}, it follows that \eqref{1.4} and \eqref{1.5} hold true
for $p_1:=\nu$, $p_2:=1$, $p_3:=\fz$, and $\epsilon:=0$. Thus, applying Theorem \ref{t1.2}, we then complete
the proof of Theorem \ref{t2.2}.
\end{proof}

Denote by $C^{\fz}_{\mathrm{c}}(\rn)$ the set of all \emph{infinitely differentiable functions on
$\rn$ with compact support}. Take $\Phi\in C^\fz_{\rm c}(\rn)$ such that $\mathbf{1}_{B(\mathbf{0},\,1)}\le\Phi\le\mathbf{1}_{B(\mathbf{0},\,2)}$,
where $\mathbf{0}$ denotes the \emph{origin} in $\rn$.
Let $\az\in[1,\fz)$. For any $f\in\cs'(\rn)$ and $x\in\rn$, let
$$
\wz{S}_{\az}(f)(x):=\lf[\int_0^\fz\int_\rn\Phi\lf(\frac{x-y}{\az t}\r)|\phi_t\ast f(y)|^2\,\frac{dy dt}{t^{n+1}}\r]^{\frac12}.
$$
Then it is easy to see that, for any $f\in\cs'(\rn)$ and $x\in\rn$,
\begin{align}\label{4.8}
S_\az(f)(x)\le\wz{S}_{\az}(f)(x)\le S_{2\az}(f)(x).
\end{align}

Now, we show Theorem \ref{t2.3} by using Theorem \ref{t1.2}.

\begin{proof}[Proof of Theorem \ref{t2.3}]
Let $f\in L^\fz_{\rm c}(\rn)$ and $F:=\wz{S}_{\az}(f)$. For any ball $B:=B(x_B,r_B)$ of $\rn$ with $x_B\in \rn$ and $r_B\in(0,\fz)$,
let $F_B:=\wz{S}_{\az}(f\mathbf{1}_{8B})$
and $R_B:=\wz{S}_{\az}(f-f\mathbf{1}_{8B})$. Then $|F|\le|F_B|+|R_B|$. Let $q\in(1,\fz)$ and $\nu\in(0,1)$.
From the assumption $\omega\in A_\fz(\rn)$ and Lemma \ref{l3.1}(i), we deduce that there exists an $s\in(1,\fz]$
such that $\omega\in RH_s(\rn)$.

It is known that, for any $h\in L^1(\rn)$, $\|\wz{S}_{\az}(h)\|_{L^{1,\,\fz}(\rn)}\ls\az^n\|h\|_{L^1(\rn)}$
(see, for instance, \cite[(3.1)]{l14}).
By this and Lemma \ref{l4.2}, we conclude that, for any $x_1\in B$,
\begin{align}\label{4.9}
\lf(\fint_B|F_B|^{\nu}\,dx\r)^{\frac1{\nu}}&\ls\az^n\fint_{8B}|f|\,dx\ls\cm(\az^nf)(x_1).
\end{align}
To estimate $R_B$, we borrow some ideas from the proof of \cite[Lemma 3.1]{l14}.
From the definition of $\wz{S}_{\az}(f)$, it follows that, for any $x,\,y\in B$,
\begin{align}\label{4.10}
|R_B(x)|&\le\lf[\int_0^\fz\int_\rn\lf|\Phi\lf(\frac{x-z}{\az t}\r)-\Phi\lf(\frac{y-z}{\az t}\r)\r|
\lf|\phi_t\ast (f\mathbf{1}_{\rn\backslash(8B)})(z)\r|^2\,\frac{dz dt}{t^{n+1}}\r]^{\frac12}\\
&\quad+\wz{S}_{\az}\lf(f\mathbf{1}_{\rn\backslash(8B)}\r)(y).\nonumber
\end{align}
For any ball $\wz{B}:=B(x_{\wz{B}},r_{\wz{B}})$ of $\rn$ with $x_{\wz{B}}\in \rn$ and $r_{\wz{B}}\in(0,\fz)$, let
$$T(\wz{B}):=\lf\{(z,t)\in\rr^{n+1}_+:=\rn\times(0,\fz):\ z\in \wz{B},\
t\in(0,r_{\wz{B}})\r\}.$$
Then, by the assumption $\phi\in\cs(\rn)$, we find that, for any given $\epsilon\in(0,\fz)$,
and for any $(z,t)\in T(2B)$ and any $x_1\in B$,
\begin{align}\label{4.11}
\lf|\phi_t\ast (f\mathbf{1}_{\rn\backslash(8B)})(z)\r|&\ls \int_{\rn\backslash(8B)}
\frac{t^\epsilon|f(\xi)|}{(t+|z-\xi|)^{n+\epsilon}}\,d\xi\\
&\ls\lf(\frac{t}{r_B}\r)^\epsilon\sum_{k=0}^\fz\frac{1}{2^{k\epsilon}}\frac{1}{|2^kB|}\int_{2^kB}|f(\xi)|\,d\xi
\ls\lf(\frac{t}{r_B}\r)^\epsilon\cm(f)(x_1),\nonumber
\end{align}
which further implies that, for any $x,\,y,\,x_1\in B$,
\begin{align}\label{4.12}
&\lf[\int_0^{2r_B}\int_{2B}\lf|\Phi\lf(\frac{x-z}{\az t}\r)-\Phi\lf(\frac{y-z}{\az t}\r)\r|
\lf|\phi_t\ast (f\mathbf{1}_{\rn\backslash(8B)})(z)\r|^2\,\frac{dz dt}{t^{n+1}}\r]^{\frac12}\\
&\quad\ls\lf[\int_0^{2r_B}\int_{2B}\lf|\Phi\lf(\frac{x-z}{\az t}\r)-\Phi\lf(\frac{y-z}{\az t}\r)\r|
\lf(\frac{t}{r_B}\r)^{2\epsilon}\,\frac{dz dt}{t^{n+1}}\r]^{\frac12}\cm(f)(x_1)\nonumber\\
&\quad\ls\lf[\int_0^{2r_B}(\az t)^n
\lf(\frac{t}{r_B}\r)^{2\epsilon}\,\frac{dz dt}{t^{n+1}}\r]^{\frac12}\cm(f)(x_1)
\ls\az^{\frac{n}2}\cm(f)(x_1).\nonumber
\end{align}
Moreover, it is easy to see that, for any $x,\,y\in B$,
\begin{align}\label{4.13}
&\lf[\int_{\rr^{n+1}_+\backslash T(2B)}\lf|\Phi\lf(\frac{x-z}{\az t}\r)-\Phi\lf(\frac{y-z}{\az t}\r)\r|\lf|\phi_t\ast
(f\mathbf{1}_{\rn\backslash(8B)})(z)\r|^2\,\frac{dz dt}{t^{n+1}}\r]^{\frac12}\\
&\quad\le\sum_{k=1}^\fz\lf[\int_{T(2^{k+1}B)\backslash T(2^kB)}\lf|\Phi\lf(\frac{x-z}{\az t}\r)
-\Phi\lf(\frac{y-z}{\az t}\r)\r|\lf|\phi_t\ast (f\mathbf{1}_{\rn\backslash(8B)})(z)\r|^2
\,\frac{dz dt}{t^{n+1}}\r]^{\frac12}.\nonumber
\end{align}

Fix $k\in\nn$ and assume that $(z,t)\in T(2^{k+1}B)\backslash T(2^kB)$. If $z\in 2^kB$, then $t\ge 2^k r_B$. Moreover,
if $z\in (2^{k+1}B)\backslash(2^kB)$, then, for any $x\in B$, $|x-z|\ge(2^k-1)r_B$. If $t<\frac{2^k-1}{2\az}r_B$, then
$\frac{|x-z|}{\az t}>2$ and $\frac{|y-z|}{\az t}>2$, and hence
$$\lf|\Phi\lf(\frac{x-z}{\az t}\r)-\Phi\lf(\frac{y-z}{\az t}\r)\r|=0.
$$
Let $t\ge\frac{2^k-1}{2\az}r_B$. Then $t\ge\frac{2^{k-2}}{\az}r_B$, which, together with the estimate
$$\lf|\Phi\lf(\frac{x-z}{\az t}\r)-\Phi\lf(\frac{y-z}{\az t}\r)\r|\le\frac{2r_B}{\az t}\|\nabla\Phi\|_{L^\fz(\rn)},
$$
further implies that, for any $x,\,y\in B$,
\begin{align}\label{4.14}
\lf|\Phi\lf(\frac{x-\cdot}{\az t}\r)-\Phi\lf(\frac{y-\cdot}{\az t}\r)\r|\mathbf{1}_{T(2^{k+1}B)\backslash T(2^kB)}
\ls\frac{r_B}{\az t}\mathbf{1}_{\{z\in 2^{k+1}B,\,t\in[2^{k-2}r_B/\az,2^{k+1}r_B]\}}.
\end{align}
From \eqref{4.14}, we deduce that, for any $x,\,y\in B$,
\begin{align}\label{4.15}
&\int_{T(2^{k+1}B)\backslash T(2^kB)}\lf|\Phi\lf(\frac{x-z}{\az t}\r)-\Phi\lf(\frac{y-z}{\az t}\r)\r|
\lf|\phi_t\ast (f\mathbf{1}_{\rn\backslash(8B)})(z)\r|^2\,\frac{dz dt}{t^{n+1}}\\
&\quad\ls\frac{r_B}{\az}\int_{2^{k-2}r_B/\az}^{2^{k+1}r_B}\int_{2^{k+1}B}\lf|\phi_t\ast
(f\mathbf{1}_{\rn\backslash(8B)})(z)\r|^2\,\frac{dz dt}{t^{n+2}}\nonumber\\
&\quad\ls\frac{r_B}{\az}\int_{2^{k-2}r_B/\az}^{2^{k+1}r_B}\int_{2^{k+1}B}\lf|\phi_t\ast
(f\mathbf{1}_{(2^{k+2}B)\backslash(8B)})(z)\r|^2\,\frac{dz dt}{t^{n+2}}\nonumber\\
&\quad\quad+\frac{r_B}{\az}\int_{2^{k-2}r_B/\az}^{2^{k+1}r_B}\int_{2^{k+1}B}\lf|\phi_t\ast
(f\mathbf{1}_{\rn\backslash(2^{k+2}B)})(z)\r|^2\,\frac{dz dt}{t^{n+2}}\nonumber\\
&\quad =:\mathrm{I}_1+\mathrm{I}_2.\nonumber
\end{align}
Using the Minkowski inequality, we know that
\begin{align*}
\mathrm{I}_1\ls\frac{r_B}{\az}\lf[\int_{2^{k+2}B}|f(\xi)|
\lf\{\int_{2^{k-2}r_B/\az}^{2^{k+1}r_B}\int_{2^{k+1}B}|\phi_t(z-\xi)|^2\,\frac{dz dt}{t^{n+2}}\r\}^{\frac12}\,d\xi\r]^2,
\end{align*}
which, combined with the estimate that, for any $\xi\in\rn$,
$$\int_{2^{k+1}B}|\phi_t(z-\xi)|^2\,dz\le\frac{\|\phi\|_{L^\fz(\rn)}}{t^n}\|\phi_t\|_{L^1(\rn)}=
\frac{\|\phi\|_{L^\fz(\rn)}\|\phi\|_{L^1(\rn)}}{t^n},
$$
implies that, for any $x_1\in B$,
\begin{align}\label{4.16}
\mathrm{I}_1\ls\frac{r_B}{\az}\lf[\int_{2^{k+2}B}|f(\xi)|\,d\xi\r]^2\int_{2^{k-2}r_B/\az}^{2^{k+1}r_B}\,\frac{dt}{t^{2n+2}}
\ls\az^{2n}2^{-k}\lf[\cm(f)(x_1)\r]^2.
\end{align}
Furthermore, similarly to the proof of \eqref{4.11}, we find that, for any given $\epsilon\in(0,\fz)$,
and for any $(z,t)\in T(2^{k+1}B)$ and any $x_1\in B$,
\begin{align*}
\lf|\phi_t\ast (f\mathbf{1}_{\rn\backslash(2^{k+2}B)})(z)\r|
\ls2^{-k\epsilon}\lf(\frac{t}{r_B}\r)^\epsilon\cm(f)(x_1),
\end{align*}
which further implies that, for any $x_1\in B$,
\begin{align}\label{4.17}
\mathrm{I}_2\ls\frac{r_B}{\az}\lf[\int_{2^{k-2}r_B/\az}^{2^{k+1}r_B}\int_{2^{k+1}B}2^{-2k\epsilon}\lf(\frac{t}{r_B}\r)^{2\epsilon}
\,\frac{dz dt}{t^{n+2}}\r]\lf[\cm(f)(x_1)\r]^2
\ls2^{-k}\az^{n-2\epsilon}\lf[\cm(f)(x_1)\r]^2.
\end{align}
Thus, by \eqref{4.15}, \eqref{4.16}, and \eqref{4.17}, we conclude that,
for any $x,\,y,\,x_1\in B$,
\begin{align*}
\lf[\int_{T(2^{k+1}B)\backslash T(2^kB)}\lf|\Phi\lf(\frac{x-z}{\az t}\r)-\Phi\lf(\frac{y-z}{\az t}\r)\r|\lf|\phi_t\ast (f\mathbf{1}_{\rn\backslash(8B)})(z)\r|^2\,\frac{dz dt}{t^{n+1}}\r]^{\frac12}
\ls2^{-\frac k2}\az^{n}\cm(f)(x_1),
\end{align*}
which, together with \eqref{4.10}, \eqref{4.12}, and \eqref{4.13}, further implies that, for any $x,\,x_1\in B$,
\begin{align}\label{4.18}
|R_B(x)|&\ls\inf_{y\in B}\lf\{\wz{S}_{\az}\lf(f\mathbf{1}_{\rn\backslash(8B)}\r)(y)\r\}
+\sum_{k=0}^\fz2^{-\frac k2}\az^{n}\cm(f)(x_1)\\
&\ls\lf(\fint_B|R_B|^{\nu}\,dy\r)^{\frac{1}{\nu}}+\az^{n}\cm(f)(x_1).\nonumber
\end{align}
From \eqref{4.9} and \eqref{4.18}, it follows that, for any $x_1,\,x_2\in B$,
\begin{align}\label{4.19}
\lf\|R_B\r\|_{L^\fz(B)}
&\ls\lf(\fint_B|R_B|^{\nu}\,dy\r)^{\frac{1}{\nu}}+\az^{n}\cm(f)(x_1)\\
&\ls\lf(\fint_B|F|^{\nu}\,dy\r)^{\frac{1}{\nu}}+\lf(\fint_B|F_B|^{\nu}\,dy\r)^{\frac{1}{\nu}}+\cm(\az^{n}f)(x_1)\nonumber\\
&\ls\lf[\cm\lf(|F|^{\nu}\r)(x_2)\r]^{\frac1{\nu}}+\cm\lf(\az^n f\r)(x_1).\nonumber
\end{align}
Then, by \eqref{4.9} and \eqref{4.19}, we know that \eqref{1.4} and \eqref{1.5} hold true
for $p_1:=\nu$, $p_2:=1$, $p_3:=\fz$, and $\epsilon:=0$. Thus, applying Theorem \ref{t1.2} and \eqref{4.8}, we complete
the proof of Theorem \ref{t2.3}.
\end{proof}

To prove Theorem \ref{t2.4}, we need the following Kolmogorov type inequality for the fractional integral operator $I_\az$.

\begin{lemma}\label{l4.4}
Let $\az\in(0,n)$, $\nu\in(0,1)$, and the measurable subset $E$ of $\rn$ satisfy $|E|<\fz$.
Then there exists a positive constant $C$, depending only on $n$, $\az$, and $\nu$, such that, for any
$f\in L^1(\rn)$,
$$\int_E|I_\az(f)|^\nu\,dx\le C|E|^{1-\frac{\nu(n-\az)}{n}}\|f\|_{L^1(\rn)}^\nu.
$$
\end{lemma}
\begin{proof}
Let $f\in L^1(\rn)$. From the well-known boundedness of $I_\az$ from $L^1(\rn)$ to $L^{\frac{n}{n-\az},\,\fz}(\rn)$
(see, for instance, \cite[p.\,119, Theorem 1]{s70}), we deduce that there exists a positive constant $C$, depending only on $n$ and $\az$,
such that $\|I_\az(f)\|_{L^{\frac{n}{n-\az},\,\fz}(\rn)}\le C\|f\|_{L^1(\rn)}$. By this, we conclude that
\begin{align*}
\int_E|I_\az(f)|^\nu\,dx&=\nu\int_0^\fz\lz^{\nu-1}|\{x\in E:\ |I_\az(f)(x)|>\lz\}|\,d\lz\\
&\le\nu\int_0^\fz\lz^{\nu-1}\min\lf\{|E|,\,\lf[\frac{C}{\lz}\|f\|_{L^1(\rn)}\r]^{\frac{n}{n-\az}}\r\}d\lz\\
&=\nu\int_0^{\frac{C\|f\|_{L^1(\rn)}}{|E|^{(n-\az)/n}}}\lz^{\nu-1}|E|\,d\lz+
\nu\int_{\frac{C\|f\|_{L^1(\rn)}}{|E|^{(n-\az)/n}}}^\fz C^{\frac{n}{n-\az}}\lz^{\nu-1-\frac{n}{n-\az}}
\|f\|_{L^1(\rn)}^{\frac{n}{n-\az}}\,d\lz\\
&\ls|E|^{1-\nu(\frac{n-\az}{n})}\|f\|_{L^1(\rn)}^\nu.
\end{align*}
This finishes the proof of Lemma \ref{l4.4}.
\end{proof}

Next, we prove Theorem \ref{t2.4} via using Theorem \ref{t1.3} and Lemma \ref{l4.4}.

\begin{proof}[Proof of Theorem \ref{t2.4}]
Let $f\in L^\fz_{\rm c}(\rn)$ and $F:=I_{\az}(f)$. For any ball $B:=B(x_B,r_B)$ of $\rn$ with $x_B\in \rn$ and $r_B\in(0,\fz)$,
let $F_B:=I_{\az}(f\mathbf{1}_{8B})$ and $R_B:=I_{\az}(f-f\mathbf{1}_{8B})$. Then $|F|\le|F_B|+|R_B|$.
Let $\nu\in(0,1)$, $p\in(1,\frac{n}{\az})$, and $q\in(1,\fz)$ satisfy $\frac{1}{q}=\frac{1}{p}-\frac{\az}{n}$.
From the assumption $\omega\in A_\fz(\rn)$ and Lemma \ref{l3.1}(i), it follows that there exists an $s\in(1,\fz]$
such that $\omega\in RH_s(\rn)$.

By Lemma \ref{l4.4}, we find that
$$\lf[\fint_B|I_\az(f\mathbf{1}_{8B})|^{\nu}\,dx\r]^{\frac1{\nu}}\ls|B|^{\frac{\az}{n}-1}\int_{8B}|f|\,dx,
$$
which further implies that, for any $x_1\in B$,
\begin{align}\label{4.20}
\lf(\fint_B|F_B|^{\nu}\,dx\r)^{\frac1{\nu}}\ls|8B|^{\frac{\az}{n}}\fint_{8B}|f|\,dx
\ls\cm_{\az/n}(f)(x_1).
\end{align}

For any $x,\,y\in\rn$ with $x\neq y$, let $K(x,y):=\frac{1}{|x-y|^{n-\az}}$.
Then it is easy to see that, for any $x,\,y,\,h\in\rn$ satisfying $0<|h|\le\frac{|x-y|}{2}$,
$$|K(x+h,y)-K(x,y)|\ls|h||x-y|^{\az-n-1},
$$
which implies that, for any $x,\,y,\,x_1\in B$,
\begin{align*}
&\lf|I_\az(f\mathbf{1}_{\rn\backslash(8B)})(x)-I_\az(f\mathbf{1}_{\rn\backslash(8B)})(y)\r|\\ \nonumber
&\quad\ls\int_{\rn\backslash(8B)}|K(x,z)-K(y,z)||f(z)|\,dz\ls|x-y|
\int_{\rn\backslash(8B)}\frac{|f(z)|}{|x-z|^{n+1-\az}}\,dz\\ \nonumber
&\quad\ls r_B\sum_{j=2}^\fz\frac{1}{(2^jr_B)^{n+1-\az}}\int_{2^{j+1}B}|f(y)|\,dy
\ls\sum_{j=2}^\fz\frac{1}{2^j}\lf(|2^{j+1}B|^{\frac{\az}{n}}\fint_{2^{j+1}B}|f|\,dz\r)\\ \nonumber
&\quad\ls\sum_{j=2}^\fz\frac{1}{2^j}\cm_{\az/n}(f)(x_1)
\ls\cm_{\az/n}(f)(x_1).
\end{align*}
From this and \eqref{4.20}, we deduce that, for any $x_1,\,x_2\in B$,
\begin{align}\label{4.21}
\lf\|R_B\r\|_{L^\fz(B)}&\ls
\inf_{y\in B}\lf\{\lf|I_\az(f\mathbf{1}_{\rn\backslash(8B)})(y)\r|\r\}+\cm_{\az/n}(f)(x_1)\\
&\ls\lf[\fint_B|I_\az(f)|^\nu\,dz\r]^{\frac1\nu}+\lf[\fint_B|I_\az(f\mathbf{1}_{8B})|^{\nu}\,dz\r]^{\frac1{\nu}}
+\cm_{\az/n}(f)(x_1)\nonumber\\
&\ls\lf(\fint_B|F|^{\nu}\,dz\r)^{\frac1{\nu}}+\cm_{\az/n}(f)(x_1)
\ls\cm_{\az/n}(f)(x_1)+\lf[\cm(|F|^{\nu})(x_2)\r]^{\frac1{\nu}}.\nonumber
\end{align}
Then, by \eqref{4.20} and \eqref{4.21}, we conclude that \eqref{1.4} and \eqref{1.5} hold true
for $p_1:=\nu,\,p_2:=1$, $p_3:=\fz$, and $\epsilon:=0$. Thus, by Theorem \ref{t1.3}, we then complete
the proof of Theorem \ref{t2.4}.
\end{proof}

\section{Proofs of Theorems \ref{t2.5}, \ref{t2.6}, and \ref{t2.7}\label{s5}}

In this section, we give the proofs of Theorems \ref{t2.5}, \ref{t2.6}, and \ref{t2.7}
via using Theorems \ref{t1.2} and \ref{t1.3}.

To show Theorem \ref{t2.5}, we need the following Lemma \ref{l5.1}, which was established in
\cite[Lemma 5.3]{am06} (see also \cite{a07}).

\begin{lemma}\label{l5.1}
Let $L$ be as in \eqref{2.3} and $p_0,\,q_0\in(q_-(L),q_+(L))$ with $p_0\le q_0$. Assume that $B:=B(x_B,r_B)$
is a ball of $\rn$ with $x_B\in \rn$ and $r_B\in(0,\fz)$. Then, for any $m\in\nn$ large enough and $f\in L^\fz_{\rm c}(\rn)$,
\begin{align*}
\lf[\fint_B\lf|\nabla L^{-1/2}\lf(I-e^{-r_B^2L}\r)^m(f)\r|^{p_0}\,dx\r]^{\frac{1}{p_0}}
\le\sum_{j=1}^\fz g_1(j)\lf(\fint_{2^{j+1}B}|f|^{p_0}\,dx\r)^{\frac{1}{p_0}},
\end{align*}
and, for any $k\in\{1,\,\ldots,\,m\}$ and $f\in L^{p_0}(\rn)$ satisfying $\nabla f\in L^{p_0}(\rn)$,
\begin{align*}
\lf[\fint_B\lf|\nabla e^{-kr_B^2L}(f)\r|^{q_0}\,dx\r]^{\frac{1}{q_0}}
\le\sum_{j=1}^\fz g_2(j)\lf(\fint_{2^{j+1}B}|\nabla f|^{p_0}\,dx\r)^{\frac{1}{p_0}},
\end{align*}
where, for any $j\in\nn$, $g_1(j):=C_{(m)}2^{j\tz}4^{-mj}$ and $g_2(j):=C_{(m)}2^{j}(\sum_{\ell=j}^\fz2^{\ell\tz}e^{-\az 4^\ell})$
for some $\tz,\,\az\in(0,\fz)$. Here $C_{(m)}$ is a positive constant depending only on $m$.
\end{lemma}

Now, we prove Theorem \ref{t2.5} via using Lemma \ref{l5.1} and Theorem \ref{t1.2}.

\begin{proof}[Proof of Theorem \ref{t2.5}]
From the assumption $v^{1-(\frac{q}{q_-(L)})'}\in A_\fz(\rn)$, it follows that there exists an $r_0\in(1,\fz)$
such that $v^{1-(\frac{q}{q_-(L)})'}\in RH_{r_0}(\rn)$. Let $p_0\in(q_-(L),q_+(L))$ be such that
$(1-(\frac{q}{q_-(L)})')r_0=1-(\frac{q}{p_0})'$. This, combined with the assumption \eqref{2.4},
further implies that, for any ball $B$ of $\rn$,
\begin{align*}
&\lf\{\lf[\fint_B \omega\,dx\r]^{\frac{p_0}{q}}
\lf[\fint_B v^{1-(\frac{q}{p_0})'}\,dx\r]^{1/(\frac{q}{p_0})'}\r\}^{\frac{q_-{(L)}}{p_0}}\\ \nonumber
&\quad=\lf[\fint_B \omega\,dx\r]^{\frac{q_-{(L)}}{q}}
\lf[\fint_B v^{[1-(\frac{q}{q_{(L)}})']r_0}\,dx\r]^{\frac{1}{r_0}\frac{1}{(\frac{q}{q_-{(L)}})'}}\\ \nonumber
&\quad\ls\lf[\fint_B \omega\,dx\r]^{\frac{q_-{(L)}}{q}}\lf[\fint_B v^{1-(\frac{q}{q_-{(L)}})'}\,dx\r]^{\frac1{(\frac{q}{q_-{(L)}})'}}
\ls\lf[\omega,v^{1-(\frac{q}{q_-(L)})'}\r]_{A_{\frac{q}{q_-(L)}}(\rn)}^{\frac{q_-{(L)}}{q}}.
\end{align*}
By this, we conclude that
\begin{align*}
\lf[\omega,v^{1-(\frac{q}{p_0})'}\r]_{A_{\frac{q}{p_0}}(\rn)}
=\sup_{B\subset\rn}\lf[\fint_B \omega\,dx\r]
\lf[\fint_B v^{1-(\frac{q}{p_0})'}\,dx\r]^{\frac{q}{p_0}-1}\ls\lf[\omega,v^{1-(\frac{q}{q_-(L)})'}\r]_{A_{\frac{q}{q_-(L)}}(\rn)}<\fz.
\end{align*}
Moreover, from the assumption $\omega\in RH_s(\rn)$ and $s>(\frac{q_+(L)}{q})'$, we deduce that there exists a $q_0\in(q,q_+(L))$
such that $\omega\in RH_s(\rn)$ and $s>(\frac{q_0}{q})'$.

Let $m\in\nn$ be large enough such that $\sum_{j=1}^\fz g_1(j)\ls1$ and $\sum_{j=1}^\fz g_2(j)\ls1$,
where, for any $j\in\nn$, $g_1(j)$ and $g_2(j)$ are as in Lemma \ref{l5.1}. Assume that $B:=B(x_B,r_B)$,
with $x_B\in\rn$ and $r_B\in(0,\fz)$, is a ball of $\rn$ and $f\in L^\fz_{\rm c}(\rn)$.
Let $F:=\nabla L^{-1/2}(f)$, $F_B:=\nabla L^{-1/2}(I-e^{-r_B^2}L)^m(f)$,
and $R_B:=\nabla L^{-1/2}[I-(I-e^{-r_B^2}L)^m](f)$. Then $|F|\le |F_B|+|R_B|$ on $B$. By Lemma \ref{l5.1},
we find that, for any $x_1,\,x_2\in B$,
\begin{align}\label{5.1}
\lf(\fint_B|F_B|^{p_0}\,dx\r)^{\frac{1}{p_0}}\ls
\sum_{j=1}^\fz g_1(j)\lf(\fint_{2^{j+1}B}|f|^{p_0}\,dx\r)^{\frac{1}{p_0}}\ls\lf[\cm(|f|^{p_0})(x_1)\r]^{\frac{1}{p_0}}
\end{align}
and
\begin{align}\label{5.2}
\lf(\fint_B|R_B|^{q_0}\,dx\r)^{\frac{1}{q_0}}&\ls\sum_{k=1}^m
\lf[\fint_B\lf|\nabla e^{-kr_B^2L} L^{-1/2}(f)\r|^{q_0}\,dx\r]^{\frac{1}{q_0}}\\
&\ls\sum_{j=1}^\fz g_2(j)\lf[\fint_{2^{j+1}B}|\nabla L^{-1/2}(f)|^{p_0}\,dx\r]^{\frac{1}{p_0}}\nonumber\\
&\ls\sum_{j=1}^\fz g_2(j)\lf[\cm(|\nabla L^{-1/2}(f)|^{p_0})(x_2)\r]^{\frac{1}{p_0}}
\ls\lf[\cm(|F|^{p_0})(x_2)\r]^{\frac{1}{p_0}}.\nonumber
\end{align}
From \eqref{5.1} and \eqref{5.2}, it follows that \eqref{1.4} and \eqref{1.5} hold true
for $p_1=p_2:=p_0$, $p_3:=q_0$, and $\epsilon:=0$. Thus, applying Theorem \ref{t1.2}, we then complete
the proof of Theorem \ref{t2.5}.
\end{proof}

To prove Theorem \ref{t2.6} by using Theorem \ref{t1.2}, we need the following Lemma \ref{l5.2},
which was obtained in \cite[pp.\,732-734]{am06}.

\begin{lemma}\label{l5.2}
Let $L$ be as in \eqref{2.3} and $p_0,\,q_0\in(q_-(L),q_+(L))$ with $p_0\le q_0$. Assume that $B:=B(x_B,r_B)$,
with $x_B\in \rn$ and $r_B\in(0,\fz)$, is a ball of $\rn$. Then, for any $m\in\nn$ large
enough and $f\in L^\fz_{\rm c}(\rn)$,
\begin{align*}
\lf[\fint_B\lf|G_L\lf(I-e^{-r_B^2L}\r)^m(f)\r|^{p_0}\,dx\r]^{\frac{1}{p_0}}
\le\sum_{j=1}^\fz g_1(j)\lf(\fint_{2^{j+1}B}|f|^{p_0}\,dx\r)^{\frac{1}{p_0}},
\end{align*}
and, for any $k\in\{1,\,\ldots,\,m\}$,
\begin{align*}
\lf[\fint_B\lf|G_L\lf(e^{-kr_B^2L}(f)\r)\r|^{q_0}\,dx\r]^{\frac{1}{q_0}}
\le\sum_{j=1}^\fz g_2(j)\lf[\fint_{2^{j+1}B}|G_L(f)|^{p_0}\,dx\r]^{\frac{1}{p_0}},
\end{align*}
where, for any $j\in\nn$, $g_1(j):=C_{(m)}2^{j\tz}4^{-mj}$ and $g_2(j):=C_{(m)}2^{j}(\sum_{\ell=j}^\fz2^{\ell\tz}e^{-\az 4^\ell})$
for some $\tz,\,\az\in(0,\fz)$. Here $C_{(m)}$ is a positive constant depending only on $m$.
\end{lemma}

\begin{proof}[Proof of Theorem \ref{t2.6}]
Replacing Lemma \ref{l5.1} by Lemma \ref{l5.2} and repeating the proof of Theorem \ref{t2.5},
we then complete the proof of Theorem \ref{t2.6}.
\end{proof}

To show Theorem \ref{t2.7} by using Theorem \ref{t1.3}, we need the following Lemma \ref{l5.3},
which was established in \cite[Lemma 3.2]{am08}.

\begin{lemma}\label{l5.3}
Let $L$ be as in \eqref{2.3}, $\az\in(0,n)$, and $p_-(L)<p_0<s_0<q_0<p_+(L)$ satisfy $\frac{1}{p_0}-\frac{1}{s_0}=
\frac{\az}{n}$. Assume that $B:=B(x_B,r_B)$, with $x_B\in\rn$ and $r_B\in(0,\fz)$, is a ball of $\rn$.
Then, for any $m\in\nn$ large enough and $f\in L^\fz_{\rm c}(\rn)$,
\begin{align*}
\lf[\fint_B\lf|L^{-\az/2}\lf(I-e^{-r_B^2L}\r)^m(f)\r|^{s_0}\,dx\r]^{\frac{1}{s_0}}
\le\sum_{j=1}^\fz g_1(j)(2^{j+1}r_B)^{\az}\lf(\fint_{2^{j+1}B}|f|^{p_0}\,dx\r)^{\frac{1}{p_0}},
\end{align*}
and, for any $k\in\{1,\,\ldots,\,m\}$,
\begin{align*}
\lf[\fint_B\lf|L^{-\az/2}\lf(e^{-kr_B^2L}(f)\r)\r|^{q_0}\,dx\r]^{\frac{1}{q_0}}
\le\sum_{j=1}^\fz g_2(j)\lf[\fint_{2^{j+1}B}\lf|L^{-\az/2}(f)\r|^{s_0}\,dx\r]^{\frac{1}{s_0}},
\end{align*}
where, for any $j\in\nn$, $g_1(j):=C_{(m)}2^{-j(2m-n/s_0)}$ and $g_2(j):=C_{(m)}e^{-c4^j}$
for some positive constant $c$. Here $C_{(m)}$ is a positive constant depending only on $m$.
\end{lemma}

Next, we prove Theorem \ref{t2.7} via using Lemma \ref{l5.3} and Theorem \ref{t1.3}.

\begin{proof}[Proof of Theorem \ref{t2.7}]
By the assumption $v^{1-(\frac{p}{p_-(L)})'}\in A_\fz(\rn)$, we conclude that there exists an $r_0\in(1,\fz)$
such that $v^{1-(\frac{p}{p_-(L)})'}\in RH_{r_0}(\rn)$. Let $p_0\in(p_-(L),p_+(L))$ be such that
$[1-(\frac{p}{p_-(L)})']r_0=1-(\frac{p}{p_0})'$ and $\gz:=\frac{\az p_0}{n}$. From this and the assumption \eqref{2.5},
we deduce that, for any ball $B$ of $\rn$,
\begin{align*}
&\lf[\fint_B \omega\,dx\r]^{\frac{p_0}{q}}
\lf[\fint_B v^{1-(\frac{p}{p_0})'}\,dx\r]^{\frac{1}{(\frac{p}{p_0})'}}\\
&\quad=\lf[\fint_B \omega\,dx\r]^{\frac{p_0}{q}}
\lf[\fint_B v^{[1-(\frac{p}{p_{(L)}})']r_0}\,dx\r]^{\frac{1}{r_0}\frac{1}{(\frac{q}{p_-(L)})'}\frac{p_0}{p_-(L)}}\\ \nonumber
&\quad\ls\lf\{\lf[\fint_B \omega\,dx\r]^{\frac{p_-{(L)}}{q}}\lf[\fint_B v^{1-(\frac{p}{p_-{(L)}})'}\,dx\r]^{\frac{1}{(\frac{p}{p_-{(L)}})'}}
\r\}^{\frac{p_0}{p_-(L)}}
\ls\lf[\omega,v^{1-(\frac{p}{p_-(L)})'}\r]_{A_{\frac{p_-(L)}{q}(1-\frac{\az p_-(L)}{n})}(\rn)}^{\frac{p_0}{q}},
\end{align*}
which further implies that
\begin{align*}
\lf[\omega,v^{1-(\frac{p}{p_0})'}\r]_{A_{\frac{q}{p_0}(1-\gz)}(\rn)}&=
\sup_{B\subset\rn}\lf[\fint_B \omega\,dx\r]
\lf[\fint_B v^{1-(\frac{p}{p_0})'}\,dx\r]^{\frac{q}{p_0}(1-\gz)-1}\\
&\ls\lf[\omega,v^{1-(\frac{p}{p_-(L)})'}\r]_{A_{\frac{p_-(L)}{q}(1-\frac{\az p_-(L)}{n})}(\rn)}<\fz.
\end{align*}
Moreover, by the assumption $\omega\in RH_s(\rn)$ and $s>(\frac{q_+(L)}{q})'$, we know that there exists a $q_0\in(q,q_+(L))$
such that $\omega\in RH_s(\rn)$ and $s>(\frac{q_0}{q})'$. Let $s_0\in(p_0,p_+(L))$ satisfy $\frac{1}{p_0}-\frac{1}{s_0}=
\frac{\az}{n}$.

Take $m\in\nn$ large enough such that $\sum_{j=1}^\fz g_1(j)\ls1$ and $\sum_{j=1}^\fz g_2(j)\ls1$,
where, for any $j\in\nn$, $g_1(j)$ and $g_2(j)$ are as in Lemma \ref{l5.3}. Assume that $B:=B(x_B,r_B)$,
with $x_B\in\rn$ and $r_B\in(0,\fz)$, is a ball of $\rn$ and $f\in L^\fz_{\rm c}(\rn)$.
Let $F:=L^{-\az/2}(f)$, $F_B:=L^{-\az/2}(I-e^{-r_B^2}L)^m(f)$,
and $R_B:=L^{-\az/2}[I-(I-e^{-r_B^2}L)^m](f)$. Then $|F|\le |F_B|+|R_B|$ on $B$. From Lemma \ref{l5.3},
it follows that, for any $x_1,\,x_2\in B$,
\begin{align}\label{5.3}
\lf(\fint_B|F_B|^{s_0}\,dx\r)^{\frac{1}{s_0}}\ls
\sum_{j=1}^\fz g_1(j)(2^{j+1}r_B)^{\az}\lf(\fint_{2^{j+1}B}|f|^{p_0}\,dx\r)^{\frac{1}{p_0}}
\ls\lf[\cm_{\az p_0/n}(|f|^{p_0})(x_1)\r]^{\frac{1}{p_0}}
\end{align}
and
\begin{align}\label{5.4}
\lf(\fint_B|R_B|^{q_0}\,dx\r)^{\frac{1}{q_0}}&\ls\sum_{k=1}^m
\lf[\fint_B\lf|L^{-\az/2}\lf(e^{-kr_B^2L}(f)\r)\r|^{q_0}\,dx\r]^{\frac{1}{q_0}}\\
&\ls\sum_{j=1}^\fz g_2(j)\lf[\fint_{2^{j+1}B}\lf|L^{-\az/2}(f)\r|^{s_0}\,dx\r]^{\frac{1}{s_0}}\nonumber\\
&\ls\sum_{j=1}^\fz g_2(j)\lf[\cm\lf(\lf|L^{-\az/2}(f)\r|^{s_0}\r)(x_2)\r]^{\frac{1}{s_0}}
\ls\lf[\cm(|F|^{s_0})(x_2)\r]^{\frac{1}{s_0}}.\nonumber
\end{align}
By \eqref{5.3} and \eqref{5.4}, we find that \eqref{1.4} and \eqref{1.5} hold true
for $p_1:=s_0$, $p_2:=p_0$, $p_3:=q_0$, $\gz:=\az p_0/n$, and $\epsilon:=0$. Thus, applying Theorem \ref{t1.3},
we then complete the proof of Theorem \ref{t2.7}.
\end{proof}

\smallskip

\noindent{\textbf{Acknowledgements}}\quad Zhenyu Yang would like to thank Professor Wen Yuan
for his guidance and encouragements, and some helpful discussions on the topic of this article.

\bigskip

\noindent Sibei Yang (Corresponding author)

\medskip

\noindent School of Mathematics and Statistics, Gansu Key Laboratory of Applied Mathematics
and Complex Systems, Lanzhou University, Lanzhou 730000, People's Republic of China

\smallskip

\noindent{\it E-mail:} \texttt{yangsb@lzu.edu.cn}

\bigskip

\noindent Zhenyu Yang

\medskip

\noindent Laboratory of Mathematics and Complex Systems (Ministry of Education of China),
School of Mathematical Sciences, Beijing Normal University, Beijing 100875,
People's Republic of China

\smallskip

\smallskip

\noindent{\it E-mail:} \texttt{zhenyuyang@mail.bnu.edu.cn}

\end{document}